%% file: main.tex
\documentclass[oldfontcommands, 11pt, reqno]{amsart} 

\input{macros}
\usepackage{verbatim}  
\usepackage{vector}  

\usepackage[style = alphabetic]{biblatex}
\usepackage{hyperref}
\hypersetup{linkcolor=blue, urlcolor=blue, colorlinks=true, linktocpage=true, citebordercolor = 0 1 0}  
\begin{document}

\title  {Factorization Algebras for Classical Bulk-Boundary Systems}
\author {Eugene Rabinovich}
\maketitle

\input{abstract}


\pagestyle{headings}  

\input{Chapter2} 


\addtocontents{toc}{\vspace{2em}} 

\appendix 

\input{AppendixA}	

\addtocontents{toc}{\vspace{2em}}  

\label{Bibliography}
\printbibliography
\end{document}

%% file: macros.tex
\usepackage{mathpple}
\usepackage{mathrsfs}
\usepackage{amsmath, amscd, amsthm,amssymb, thmtools, amsfonts, verbatim, breqn}
\usepackage[mathcal]{eucal}
\usepackage{relsize}
\usepackage{tikz-cd}
\usepackage[cmtip, arrow, all]{xy}
\usepackage{pb-diagram ,pb-xy}
\usepackage{graphicx}
\usepackage{subcaption}
\usepackage{epstopdf}
\usepgfmodule{shapes}
\usepackage{color}
\usepackage{ifthen}
\usetikzlibrary{patterns, arrows}

\renewcommand{\epsilon}{\varepsilon}

\newcommand{\CC}{\mathbb C}

\newcommand{\Z}{\mathbb Z}

\newcommand{\mbb}{\mathbb}

\newcommand{\ip}[1][\cdot,\cdot]{\left\langle #1 \right\rangle}

\newcommand{\R}{\mbb R}

\newcommand{\DVS}{\mathrm{DVS}}
\newcommand{\CVS}{\mathrm{CVS}}

\DeclareMathOperator{\Sym}{Sym} \DeclareMathOperator{\Hom}{Hom}

\def\cC{\mathcal C}\def\cD{\mathcal D}
\def\cF{\mathcal F}
\def\cK{\mathcal K}

\def\cS{\mathcal S}\def\cT{\mathcal T}

\def\CC{\mathbb C}
\def\HH{\mathbb H}
\def\LL{\mathbb L}

\def\RR{\mathbb R}

\def\sC{\mathscr C}
\def\sE{\mathscr E}\def\sF{\mathscr F}\def\sG{\mathscr G}
\def\sJ{\mathscr J}\def\sL{\mathscr L}
\def\sO{\mathscr O}

\def\sV{\mathscr V}

\def\fJ(E){\mathfrak E}
\def\fJ{\mathfrak J}

\def\fU{\mathfrak U}\def\fV{\mathfrak V}
\def\fg{\mathfrak g}\def\fh{\mathfrak h}

\declaretheoremstyle[
spaceabove=7pt, spacebelow=7pt,
headfont=\normalfont\bfseries,
notefont=\mdseries, notebraces={(}{)},
bodyfont=\normalfont,
postheadspace=5pt,
headpunct = .
]{thm}

\declaretheoremstyle[
spaceabove=7pt, spacebelow=7pt,
headfont=\normalfont\bfseries,
notefont=\mdseries, notebraces={(}{)},
bodyfont=\normalfont,
postheadspace=10pt,
headpunct = .
]{def}

\declaretheoremstyle[
spaceabove=4pt, spacebelow=7pt,
headfont=\itshape,
postheadspace=5pt,
headpunct = :,
postheadspace = 3pt, qed = $\lozenge$
]{rem}

\declaretheorem[numbered = yes, parent = section, style = thm]{theorem}

\declaretheorem[numbered = no, style = thm, name = Corollary A]{corA}

\declaretheorem[sibling = theorem, style = thm, name = Theorem/Definition]{thm-def}
\declaretheorem[sibling = theorem, style = thm]{proposition}
\declaretheorem[sibling = theorem, style = thm]{lemma}
\declaretheorem[sibling = theorem, style = thm]{notation}

\numberwithin{equation}{section}

\declaretheorem[sibling = theorem, style = def]{definition}

\declaretheorem[sibling = theorem, style = rem]{remark}

\declaretheorem[sibling = theorem, style = rem]{example}

\newcommand{\cinfty}{C^{\infty}}
\newcommand\Obcl[1][~]{\ifthenelse{ \equal{#1}{~}} {
  \operatorname{Obs}^{cl}
}{
  \operatorname{Obs}^{cl}(#1)
}}
\newcommand\Obq[1][~]{\ifthenelse{ \equal{#1}{~}} {
  \operatorname{Obs}^{q}
}{
  \operatorname{Obs}^{q}(#1)
}}

\DeclareMathOperator{\id}{id}

\DeclareMathOperator{\im}{im}

\newcommand{\sEb}{\sE_\partial}

\newcommand{\Eb}{E_\partial}

\newcommand{\diff}{\ell_1}
\newcommand{\Qb}{\ell_{1,\partial}}

\newcommand{\condfields}{\sE_\sL}
\newcommand{\condfieldsterm}{$\sL$-conditioned }
\newcommand{\condfieldscs}{\sE_{\sL,c}}

\newcommand{\Oloc}{\sO_{loc}(\condfields)}
\newcommand{\Olocred}{\sO_{loc,red}(\condfields)}
\newcommand{\bdyM}{\partial M}

\newcommand{\fullOloc}{\mathrm{DR}_{M,\bdyM}(E;L)}

\newcommand{\Dens}{\mathrm{Dens}}
\newcommand{\tubnhd}{T}

\newcommand{\densM}{Dens_M}

\newcommand{\densbdyM}{Dens_{\bdyM}}
\DeclareMathOperator{\cone}{cone}
\newcommand{\delsmoothdistr}{\Omega^{n,n-1}_{M,tw}}

\def\del{\partial}

\def\hotimes{{\widehat{\otimes}}}

\def\im{{\rm im}}

\newcommand{\dmdr}{\textrm{DR}^\bullet_M(D_M)}
\newcommand{\dmdrbdy}{\textrm{DR}^\bullet_{\bdyM}(D_M)}
\newcommand{\dmdrpurebdy}{\textrm{DR}^\bullet_{\bdyM}(D_{\bdyM})}
\newcommand{\dmdrrel}{\textrm{DR}^\bullet_{M,\bdyM}(D_M)}
\newcommand{\dpush}{\iota_D}
\newcommand{\dpull}{\iota^D}
\newcommand{\drrel}{\textrm{DR}(M,\bdyM)}

\newcommand{\innerhom}[2]{\underline{CVS}\left( #1, #2\right)}
\newcommand{\innerhomsym}[3]{\underline{CVS}\left( #1, #2\right)_{S_{#3}}}

\usetikzlibrary{decorations.pathmorphing}
\usetikzlibrary{decorations.markings}
\pgfarrowsdeclare{chevron}{chevron}{...}
{
\pgfsetdash{}{0pt} 
\pgfsetlinewidth{.5pt}
\pgfpathmoveto{\pgfpoint{-1pt}{0pt}}
\pgfpathlineto{\pgfpoint{-2pt}{2pt}}
\pgfpathlineto{\pgfpoint{0pt}{2pt}}
\pgfpathlineto{\pgfpoint{1pt}{0pt}}
\pgfpathlineto{\pgfpoint{0pt}{-2pt}}
\pgfpathlineto{\pgfpoint{-2pt}{-2pt}}
\pgfpathclose
\pgfusepathqstroke
}
\pgfarrowsdeclare{newdiamond}{newdiamond}{...}
{
\pgfsetdash{}{0pt} 
\pgfsetlinewidth{.5pt}
\pgfpathmoveto{\pgfpoint{0pt}{0pt}}
\pgfpathlineto{\pgfpoint{-1pt}{2pt}}
\pgfpathlineto{\pgfpoint{1pt}{2pt}}
\pgfpathlineto{\pgfpoint{0pt}{0pt}}
\pgfpathlineto{\pgfpoint{1pt}{-2pt}}
\pgfpathlineto{\pgfpoint{-1pt}{-2pt}}
\pgfpathclose
\pgfusepathqstroke
}

\tikzset{
    vector/.style={decorate, decoration={snake}, draw},
	provector/.style={decorate, decoration={snake,amplitude=2.5pt}, draw},
	antivector/.style={decorate, decoration={snake,amplitude=-2.5pt}, draw},
    fermion/.style={draw=black, postaction={decorate},
        decoration={markings,mark=at position .55 with {\arrow{>}}}},
    Ahalfedge/.style={draw=black, postaction={decorate},
        decoration={markings,mark=at position .25 with {\arrow{>}}, mark = at position 1 with {\arrow[draw=black,xshift=1.5pt]{Circle[length=3pt]}}}},
    Bhalfedge/.style={draw=black, postaction={decorate},
        decoration={markings,mark=at position .80 with {\arrow[draw=black]{>}}, mark = at position 0 with {\arrow[draw=black,xshift=1.5pt]{Circle[length=3pt]}}}},
 ABedge/.style={draw=black, postaction={decorate},
        decoration={markings,mark=at position .50 with {\arrow[draw=black]{chevron}}}},
ABfulledge/.style={draw=black, postaction={decorate},
        decoration={markings,mark = at position .50 with {\arrow[draw=black]{chevron}}, mark = at position .3 with {\arrow[draw=black]{>}}, mark = at position .75 with {\arrow[draw=black]{>}}, mark = at position .0 with {\arrow[draw=black,xshift=1.5pt]{Circle[length=3pt]}}, mark = at position 1 with {\arrow[draw=black,xshift=1.5pt]{Circle[length=3pt]}}}},
BBfulledge/.style={draw=black, postaction={decorate},
        decoration={markings,mark = at position .50 with {\arrow[draw=black]{newdiamond}}, mark = at position .3 with {\arrow[draw=black]{>}}, mark = at position .75 with {\arrow[draw=black]{<}}, mark = at position .0 with {\arrow[draw=black,xshift=1.5pt]{Circle[length=3pt]}}, mark = at position 1 with {\arrow[draw=black,xshift=1.5pt]{Circle[length=3pt]}}}},
 BBedge/.style={draw=black, postaction={decorate},
        decoration={markings,mark=at position .50 with {\arrow[draw=black]{newdiamond}}}},
    fermionbar/.style={draw=black, postaction={decorate},
        decoration={markings,mark=at position .55 with {\arrow{<}}}},
    fermionnoarrow/.style={draw=black},
    gluon/.style={decorate, draw=black,
        decoration={coil,amplitude=4pt, segment length=5pt}},
    scalar/.style={dashed,draw=black, postaction={decorate},
        decoration={markings,mark=at position .55 with {\arrow[draw=black]{>}}}},
    scalarbar/.style={dashed,draw=black, postaction={decorate},
        decoration={markings,mark=at position .55 with {\arrow[draw=black]{<}}}},
    scalarnoarrow/.style={dashed,draw=black},
    electron/.style={draw=black, postaction={decorate}/,
        decoration={markings,mark=at position .55 with {\arrow[draw=black]{>}}}},
	bigvector/.style={decorate, decoration={snake,amplitude=4pt}, draw},
	    ray/.style={draw=black, postaction={decorate},
        decoration={markings,mark=at position 1 with {\arrow[draw=black]{>}}, mark = at position 0 with {}}}
}

%% file: abstract.tex
\begin{abstract}
    
The behavior of classical and quantum field theories in the presence of a spacetime boundary is a subject of great interest in both mathematics and physics.
Notably, a bulk-boundary system was used implicitly in Kontsevich's quantization of Poisson manifolds.
In this paper, we present a mathematical definition of a certain class of classical, perturbative bulk-boundary systems; we also associate, to each classical bulk-boundary system, a $P_0$ factorization algebra of observables (which is to say, the observables carry a Poisson bracket of cohomological degree +1).
Finally, we give a description of the local functionals of a bulk-boundary system in terms of $D$-modules on the spacetime. 
The mathematical definition we present here is based heavily on the work of Costello-Gwilliam and Butson-Yoo, though the construction of factorization algebras is new; in particular, we make heavy use of the Batalin-Vilkovisky formalism, which describes the fields of a classical field theory as a (--1)-shifted symplectic derived stack.

Our description of the shifted symplectic geometry of bulk-boundary systems is schematically described as follows.
Given a spacetime manifold $M$, one may formulate a classical field theory $\mathring \sE$ on the interior $\mathring M$, i.e. $\mathring \sE$ is a (--1)-shifted formal moduli problem associated to $\mathring M$.
If one tries to extend in a na\"ive way the space $\mathring \sE$ to one $\sE$ associated to $M$, one finds that $\sE$ does not possess a (--1)-shifted symplectic structure.
Instead, one finds a 0-shifted symplectic formal moduli problem $\sEb$ associated to the boundary $\bdyM$ and a Lagrangian map $\sE\to \sEb$.
A bulk-boundary system is defined to be such a pair $(\sE,\sEb)$ together with another Lagrangian $\sL\to \sEb$.
Then, the derived intersection $\sE\times^h_{\sEb}\sL$ will have a (--1)-shifted symplectic structure.

\end{abstract}

%% file: Chapter2.tex
\tableofcontents

\section{Introduction}
The behavior of classical and quantum field theories in the presence of a spacetime boundary is a subject of great interest in both mathematics and physics.
We mention here only two examples of relevance to mathematics: the work of Fjelstad, F\"uchs, Runkel, and Schweigert \autocite{RFFS} on constructing full conformal field theories in 2 dimensions using topological quantum field theories in 3 dimensions, and the work of Kontsevich \autocite{KontPSM} on constructing a deformation quantization of the algebra of functions on a Poisson manifold using ideas from topological field theory in 2 dimensions.
In this work, we present a systematic mathematical framework for the study of perturbative classical bulk-boundary systems, with an eye towards quantization \autocite{ERthesis}.

\subsection{Summary of Results}
\subsubsection*{Factorization Algebras of Observables}
Our approach follows closely that of Costello and Gwilliam \autocite{cost, CG1, CG2}, which applies in the case that the spacetime does not have a boundary. 
In the above-cited works, the authors give a rigorous definition of a perturbative classical or quantum gauge theory, and show how to construct a factorization algebra of observables associated to any classical or quantum gauge theory.
They also show that the factorization algebra of classical observables possesses a $P_0$ structure.

In this paper, we present a mathematical definition of classical bulk-boundary systems and construct factorization algebras thereof.
In other words, we extend the results of Costello and Gwilliam for \emph{classical} (though interacting) theories to manifolds with boundary.
The corresponding extensions of the quantum results of Costello and Gwilliam have appeared in \autocite{ERthesis}, and are also found in \autocite{GRW} (for free quantum bulk-boundary systems).

\subsubsection*{Local Functionals}
Another important aspect of the formalism of Costello and Gwilliam is that these authors provide a differential graded Lie algebra of local functionals.
This Lie algebra serves two functions: it describes the deformation theory of a fixed classical BV theory, and it provides a home for order-by-order obstructions to quantization of that classical BV theory.
In practice, nearly any quantization of a field theory in the Costello-Gwilliam language uses a description of the Lie algebra of local functionals in terms of $D_M$-modules (see, for just one example of many, \autocite{bcov}).
In this paper, we also provide a detailed discussion of the complex of local functionals associated to a classical bulk-boundary system in the language of $D_M$-modules.
Unfortunately, there is no natural Lie bracket on this complex; we leave it to future work to determine whether there is a quasi-isomorphic complex that does possess a compatible Lie bracket.

\subsubsection*{Boundary Conditions for BV Theories}
We will fix throughout a manifold $M$ with boundary $\bdyM$ and we will use the letter $\iota$ to denote the inclusion $\bdyM \hookrightarrow M$. The manifold $M$ serves as the spacetime background for the theory. 
We will restrict ourselves to a certain class of field theories, namely the TNBFTs (see Definition \ref{def: tnbft}, below).  
The main ingredient in the constructions of Costello and Gwilliam at the classical level is the Batalin-Vilkovisky (BV) formalism \autocite{originalBV}, which has a geometric interpretation in terms of (--1)-shifted symplectic structures \autocite{geometryofBV}, \autocite{AKSZ}, \autocite{ptvv}.
Our aim in the present work is to maintain as much of this geometric interpretation as possible: in the end, a bulk-boundary system will also have an interpretation in terms of (--1)-shifted symplectic geometry.

In the version of the BV formalism of \autocite{cost}, \autocite{CG1}, and \autocite{CG2}, a classical perturbative BV theory on the manifold without boundary $\mathring{M}:=M\backslash \bdyM$ is specified by a $\Z$-graded vector bundle $E\to \mathring{M}$ (whose space of sections $\sE$ furnishes the space of fields for the theory) together with a $(-1)$-shifted fiberwise symplectic structure $\ip_{E}$ and a cohomological degree 0 local action functional $S$ satisfying the classical master equation $\{S,S\}=0$. 
One may easily extend $E$, $\ip_{E}$, and $S$ to $M$; 
however, the presence of the boundary $\bdyM$ makes the consideration of the classical master equation more subtle, since, in almost all cases, one uses integration by parts to verify that $\{S,S\}=0$ on $\mathring{M}$.
Put in a different way, the space of fields ceases to be (--1)-shifted symplectic once we pass from $\mathring{M}$ to $M$.  

Our solution is to impose boundary conditions on the fields (the space $\sE$) so that the boundary terms obstructing the classical master equation vanish when restricted to the fields with this boundary condition. 
We note that the imposition of boundary terms is not a novel idea.
What is notable about our approach, however, is that we impose boundary conditions in a way that is consistent with the interpretation of the BV formalism in terms of the geometry of shifted symplectic spaces.
Namely, we are careful to ensure that the imposition of the boundary conditions happens in a homotopically coherent way.
The payoff is two-fold. 
First, we find that the space of fields with the boundary condition imposed is indeed (--1)-shifted symplectic, using general facts about shifted symplectic geometry. 
Second, we guarantee that our constructions naturally take into account the gauge symmetry of the problem---there is no need to separately impose boundary conditions on the fields, ghosts, etc. 

We ask for our boundary conditions to be suitably local (see Definition \ref{def: bdycond}). 
The locality condition ensures that even after the boundary condition is imposed, the fields remain the space of global sections of a sheaf on $M$: this allows us to carry through the constructions of \autocite{CG2} almost without change, namely we can construct factorization algebras of observables by following \autocite{CG2} \emph{mutatis mutandis}.
 
\subsubsection*{Field Theories which are Topological Normal to the Boundary} 
Let us say a bit more about the class of field theories we consider. 
These are the theories which are so-called ``topological normal to the boundary'' (first defined in \autocite{butsonyoo}; see also Definition \ref{def: tnbft}). 
We will often use the acronym TNBFT to stand for ``field theory which is topological normal to the boundary.'' For a TNBFT we require that, roughly speaking, in a tubular neighborhood $\tubnhd\cong \bdyM\times [0,\epsilon)$ of the boundary $\bdyM$, the space of fields $\sE$ (the sheaf of sections of the bundle $E$ introduced above; in this introduction, we will also use the letter $\sE$ to denote the space of global sections of this sheaf) admits a decomposition
\begin{equation}
\sEb\mid_{\tubnhd}\hotimes_\beta \Omega^{\bullet}_{[0,\epsilon)},
\end{equation}
where $\sEb$ is a 0-shifted symplectic space living on the boundary $\bdyM$, and arising, like $\sE$, from a $\Z$-graded bundle $\Eb\to \bdyM$.
We require all relevant structures on $\sE$ to decompose in a natural way with respect to this product structure.
We find (cf. \ref{lem: fieldslagrangian}) that the natural map $\sE\to \sEb$ obtained by pulling back forms on $[0,\epsilon)$ to $t=0$ is a Lagrangian map.
We define a boundary condition to be a certain other type of Lagrangian in $\sEb$ (cf. Definition \ref{def: bdycond}).
A general result from derived algebraic geometry (cf. Theorem 2.9 of \autocite{ptvv}) then implies that the derived intersection $\sE\times^h_{\sEb} \sL$ carries a canonical $(-1)$-shifted symplectic structure.
Since the objects $\sE,\sEb,$ and $\sL$ are in general infinite-dimensional, the above-mentioned theorem does not apply directly.
Nevertheless, the theorem provides philosophical guidance to our constructions, and we will see that the derived intersection (in an appropriate model category) $\sE\times^h_{\sEb}\sL$ does indeed possess something like the shifted-symplectic structure present for field theories on manifolds without boundary.

A pair consisting of a TNBFT $\sE$ and a boundary condition for $\sE$ will be termed a \emph{classical bulk-boundary system}. Though this term properly applies to a much more general class of objects, we use it here to avoid bulky terminology.
The reason for restricting our attention to TNBFTs is that such theories have a prescribed simple behavior near the boundary $\bdyM$. This behavior makes it easier to verify the Weiss cosheaf condition for the factorization algebra of observables of a classical bulk-boundary system. 
This behavior also enables us to study heat kernel renormalization of bulk-boundary systems in \autocite{ERthesis}.

Important examples of TNBFTs are of course the topological theories: BF theory, Chern-Simons theory, the Poisson sigma model, and topological mechanics. 
However, even for topological theories, we may consider boundary conditions which are not topological in nature. 
For example, one may study in Chern-Simons theory the chiral Wess-Zumino-Witten boundary condition, whose definition requires the choice of a complex structure on the boundary surface.
One may also study theories which have a dependence on arbitrary geometric structures on the boundary $\bdyM$, as long as their dependence in the direction normal to $\bdyM$ is topological. Mixed BF theory, Example \ref{ex: mixedbf} is one such example; it depends on the choice of a complex structure on the boundary of a three manifold of the form $N\times \R_{\geq 0}$, where $N$ is a surface.

\subsubsection*{TNBFTs as ``Universal Bulk Theories''}
In \autocite{butsonyoo}, Butson and Yoo also defined the notion of ``degenerate field theory,'' a notion which we prefer to call ``Poisson BV theory.'' 
A Poisson BV theory is like a ``normal'' BV theory in the sense that its space of fields is described in terms of formal derived algebraic geometry; but whereas a ``normal'' BV theory has a (--1)-shifted symplectic structure, we only require that a Poisson BV theory have a shifted-Poisson structure (it is in this sense that Poisson BV theories are degenerate).
Butson and Yoo associated, to any Poisson BV theory on a manifold (without boundary) $N$, a bulk-boundary system on $N\times \RR_{\geq 0}$, which they call the ``universal bulk theory'' for the Poisson BV theory.
The universal bulk theory associated to a Poisson BV theory does satisfy the additional assumptions we impose on bulk-boundary systems.
For the sake of brevity, we do not discuss universal bulk theories at any length; however, the techniques presented here are designed to apply in particular to universal bulk theories (with their canonical boundary condition).

\subsection{Conventions}
Throughout, the spacetime manifold will be denoted $M$ and its boundary $\bdyM$. The inclusion $\bdyM\to M$ will be denoted $\iota$. If $\tubnhd$ is a tubular neighborhood of $\bdyM$ in $M$ with a specified diffeomorphism $\tubnhd\cong \bdyM\times[0,\epsilon)$, then we will use $t$ and $s$ exclusively to denote the normal coordinate in $\tubnhd$.
\begin{itemize}
\item Given a manifold $M$, $\densM$ is the bundle of densities on $M$ and $\Omega^n_{M,tw}$ is it sheaf of sections. Similarly, $\Lambda^k T^*M$ is the $k$-th exterior power of the cotangent bundle on $M$ and $\Omega^k_M$ its sheaf of sections. We also write $\cinfty_M$ to denote the sheaf of sections of the trivial line bundle on $M$. We always include a subscript on such sheaves to indicate the manifold on which they are defined (since we often deal with sheaves on $\bdyM$ as well).
\item The notation $\cinfty_{M,c}$ indicates the cosheaf of compactly-supported functions on $M$; we define $\Omega^k_{M,c}$ similarly.
\item In the sequel, whenever we use a normal-font letter (e.g. $E$) for a bundle on $M$, we use the script version of that latter (e.g. $\sE$) for the corresponding sheaf of sections. This contrasts slightly with our usage in the introduction, where we used the script letter for the space of global sections of the bundle.
\item If $V$ is a finite-rank vector space, we let $V^*$ denote its dual space. We use the same notation for the fiberwise dual of a vector bundle.
\end{itemize}

\subsection{Functional-Analytic Conventions}
Because we are interested in factorization algebras whose values are chain complexes of infinite-dimensional vector spaces, we would like to find a category of infinite-dimensional vector spaces which is abelian.
The first candidate would be the category of locally convex topological vector spaces; this category, however, is infamously not abelian.
We work instead with the category $\DVS$ of differentiable vector spaces (Definition B.2.3 of \autocite{CG1}).
$\DVS$ is an abelian category with a number of other nice properties which are detailed in Appendix B of \autocite{CG1}.
From a purely formal standpoint, $\DVS$ is a very useful category to work in; the usual techniques of homological algebra (e.g. spectral sequences) apply in $\DVS$, though they don't apply in other categories of vector spaces.
However, the objects of $\DVS$ are described as sheaves on the site of smooth manifolds, which makes direct calculations in the category a bit cumbersome.
To be able to have the best of both worlds---to use formal homological-algebraic techniques and direct computations at the same time---we also work with the auxiliary category $\CVS$ of convenient vector spaces (cf. Definition B.5.1 of \autocite{CG1}, also \autocite{krieglmichor}).
The objects of $\CVS$ are vector spaces with additional structure (a ``bornology'') and a suitable completeness property, and the morphisms of $\CVS$ are simply bounded linear maps (those preserving the bornologies).
The category $\CVS$ has a number of nice properties, but is unfortunately not abelian.
The relationship between $\CVS$ and $\DVS$ is as follows: there is a functor $\CVS\to \DVS$ which preserves limits and countable coproducts.
It emphatically does \emph{not} preserve colimits.
This means, in particular, that if one computes the cohomology of a chain complex in $\CVS$, the answer may not coincide with the cohomology of its image in $\DVS$.
We therefore take the following precautions when computing colimits in $\CVS$:
\begin{enumerate}
    \item When computing a quotient $V/W$, we first find a splitting in $\CVS$: $V=W\oplus U$. This splitting will then survive in $\DVS$, and we can identify $V/W\cong U$ in either category.
    \item When computing cohomology using explicit computations in $\CVS$, we construct a homotopy equivalence of a chain complex with its cohomology. This homotopy equivalence survives into $\DVS$, and gives a computation of the cohomology of the complex at hand in either category.
\end{enumerate}
We note that both $\CVS$ and $\DVS$ can be enriched over themselves, and we let the notation $\underline{\CVS}$ and $\underline{\DVS}$ refer to the inner hom objects in their respective categories.

For a much fuller discussion of these issues, we refer the reader to Appendix B of \autocite{CG1}, or Appendix A of \autocite{classicalarxiv}.

\subsection{Outline}

In Section \ref{sec: tnbfts}, we introduce TNBFTs, as well as the class of local boundary conditions we consider. We show that the imposition of such boundary conditions on TNBFTs is consistent with locality and gauge symmetry (Lemma \ref{lem: tildefieldsmodel}).
In Section \ref{sec: localfcnls}, we discuss a $D$-module perspective on local functionals in BV theories on manifolds with boundary.
In Section \ref{sec: FAs}, we construct the $P_0$ factorization algebra of classical observables for a bulk-boundary system.
In Section \ref{sec: classexamples}, we apply the formalism to a few simple examples.
In the Appendix, we discuss some functional analytic details necessary in the main body of the paper.

\subsection{Acknowledgements}

This work is the author's own interpretation of research conducted with Benjamin Albert. The author would like to thank Benjamin for many discussions on the subject.

The author would also like to thank Dylan Butson, Kevin Costello, Owen Gwilliam, Si Li, Peter Teichner, Brian Williams, and Philsang Yoo for the many discussions related to the work presented here.
He would also like to thank Denis Nardin for clarification on a certain point in the proof of Lemma 2.28.

This material is based upon work supported by the National Science Foundation Graduate Research Fellowship Program under Grant No. DGE 1752814.

\section{Classical Bulk-Boundary Systems}
\label{sec: tnbfts}
In this section, we introduce the basic classical field-theoretic objects with which we work, namely we introduce the notion of a classical bulk-boundary system.
We use the Batalin-Vilkovisky (BV) formalism as the natural language for quantum field theory, which encodes both the equations of motion and the symmetries of a field theory in an intrinsically homotopically invariant fashion. 
The theories we consider are called ``topological normal to the boundary'' (TNBFT), a notion which we discuss in Section \ref{sec: classTNBFTs}.  
The definition we use here is due to Butson and Yoo \autocite{butsonyoo}. For a manifold with boundary $M$, a TNBFT is a field theory on  $\mathring M$ (the interior of $M$) which has a specified behavior of the field theory near the boundary. 
As the name suggests, the solutions to the equations of motion of a TNBFT are constant along the direction normal to the boundary $\bdyM$ in some tubular neighborhood $\tubnhd\cong \bdyM\times [0,\epsilon)$ of the boundary.

In Section \ref{sec: boundcnds}, we introduce boundary conditions for TNBFTs and discuss the homotopical interpretation of the resulting bulk-boundary systems.

\subsection{Classical Field Theories on Manifolds with Boundary}

\label{sec: classTNBFTs}
In this subsection, we define a TNBFT and give examples of TNBFTs.
\subsubsection{Definitions}
Before elaborating on the definition of a TNBFT, let us first recall what is meant by a perturbative classical field theory on $\mathring M$. The definition here follows \autocite{CG2}.

\begin{definition}
\label{def: clbv}
A \textbf{perturbative classical BV theory} on $\mathring M$ consists of
\begin{enumerate}
\item a $\Z$-graded vector bundle $E\to \mathring M$,
\item sheaf maps 
\begin{equation}
\ell_k: (\sE[-1])^{\hotimes_\beta k} \to \sE[-1] 
\end{equation}
of degree $2-k$,
\item and a degree $-1$ bundle map 
\begin{equation}
\ip_{loc} : E\otimes E \to \mathrm{Dens}_{\mathring M},
\end{equation}
The objects $E,\ell_k, \ip_{loc}$ are required to satisfy the following properties:
\begin{itemize}
\item the maps $\ell_k$ turn $\sE[-1]$ into a sheaf of $L_\infty$ algebras;
\item the maps $\ell_k$ are polydifferential operators (we will call the pair $(\sE[-1],\ell_k)$ a \textbf{local $L_\infty$ algebra} on $\mathring M$;
\item the complex $(\sE,\ell_1)$ is elliptic;
\item the pairing $\ip_{loc}$  is fiberwise non-degenerate;
\item let us denote by $\ip$ the induced map of precosheaves 
\begin{equation}
\sE_c\otimes \sE_c \to \underline{\RR}
\end{equation}
given by using the fiberwise pairing $\ip_{loc}$, and then integrating the result. (Here, $\underline \RR$ is the constant pre-cosheaf assigning $\RR$ to each open subset.) We use the same notation for the pairing induced on $\sE_c[-1]$. We require that, endowed with this pairing, $\sE[-1]$ becomes a precosheaf of cyclic $L_\infty$ algebras (on $\sE_c[-1]$, the pairing has degree --3). 
\end{itemize}
\end{enumerate}
\end{definition}
We will often think of $\sE$ as a sheaf of (--1)-shifted symplectic formal moduli problems on $\mathring M$.
This is simply a way of recasting the axioms above in a more geometric language.
\begin{remark}
\label{rmk: classfcnl}
To a field theory in the above sense, we may associate the following element of 
\begin{equation}
\prod_{k\geq 1} \CVS\left(\sE_c^{\hotimes_\beta k}, \RR \right)^{S_k},
\end{equation}
(the space of functionals on $\sE$):
\begin{equation}
S(\varphi) = \sum_{k\geq 1}\frac{1}{(k+1)!} \ip[\varphi, \ell_k(\varphi,\ldots, \varphi)],
\end{equation}
which is usually denoted the \textbf{classical action functional} of the theory.
We also make the definition 
\begin{equation}
I(\varphi ) = \sum_{k\geq 2} \frac{1}{(k+1)!}\ip[\varphi, \ell_k(\varphi,\ldots, \varphi)]
\end{equation}
(i.e. $I$ remembers only the brackets of arity at least 2).
This is the \textbf{classical interaction functional} of the theory.
The cyclicity of the brackets $\ell_k$ with respect to the pairing $\ip$ guarantees that $S$ and $I$ are symmetric in their inputs.
\end{remark}

\begin{remark}
The ellipticity of the complex $(\sE, \diff)$ is not a necessary requirement from the standpoint of physics, and in fact, there is no need for this requirement in the definition of a classical BV theory. We include it here because the theory of elliptic PDE furnishes a wealth of tools which make it possible to develop a framework for quantization.
\end{remark}

We now specify the precise definition of a TNBFT. The following definition is adapted from Definitions 3.8 and 3.9 of \autocite{butsonyoo}.

\begin{definition}
\label{def: tnbft}
A \textbf{field theory  on $M$ which is topological normal to the boundary} is specified, as in Definition \ref{def: clbv}, by a $\Z$-graded bundle $E\to M$, a collection of sheaf maps $\ell_k$, and a bundle map $\ip_{loc}$. We also specify the following data:
\begin{enumerate}
\item a $\Z$-graded bundle $\Eb\to \bdyM$;
\item a collection of sheaf maps
\begin{equation}
\ell_{k,\partial} : (\sEb[-1])^{\otimes k}\to \sEb[-1];
\end{equation}
\item and a degree 0 bundle map 
\begin{equation}
\ip_{loc,\partial}: \Eb \otimes \Eb \to \densbdyM.
\end{equation}
\item  In some tubular neighborhood $\tubnhd\cong \bdyM\times [0,\epsilon)$ of $\bdyM$, an isomorphism
\begin{equation}
\phi :E\mid_{\tubnhd}\cong \Eb\boxtimes \Lambda^\bullet T^* [0,\epsilon).
\end{equation}

\end{enumerate}

We require the following to hold:
\begin{itemize}
\item The $k$-th graded summand $E^k$ of $E$ is zero for $|k|>>0$.
\item When $E$, $\ell_k$, and $\ip_{loc}$ are restricted to $\mathring M$, the resulting data satisfy the conditions to be a classical BV theory on $\mathring M$.
\item The data $(\sEb,\ell_{k,\partial}, \ip_{loc,\partial})$ satisfy all the requirements of Definition \ref{def: clbv} (with the degree $-1$ pairing $\ip_{loc}$ replaced by the degree 0 pairing $\ip_{loc,\partial}$).
\item The isomorphism $\phi$ respects all relevant structures. More precisely, we require
\begin{itemize}
\item Over $\tubnhd$, the induced isomorphism 
\begin{equation}
\varphi: \sE[-1]\mid_{\tubnhd} \cong \sEb[-1]\hotimes_\beta \Omega^\bullet_{[0,\epsilon)}
\end{equation}
is an isomorphism of (sheaves of) $L_\infty$ algebras. Here, the target of $\varphi$ has the $L_\infty$-algebra structure induced from its decomposition as a tensor product of an $L_\infty$ algebra with a commutative differential graded algebra.
\item Over $\tubnhd$, the fiberwise pairing $\ip_{loc}$ is identified with the tensor product of the pairing $\ip_{loc,\partial}$ and the wedge product pairing $\bigwedge$ on $\Lambda^\bullet T^* [0,\epsilon)$. 
\end{itemize}
\end{itemize}
\end{definition}

Before listing examples of TNBFTs, let us mention a few of their properties which follow directly from the definitions. First, however, we need to establish a definition:

\begin{definition}
Let $\iota: \bdyM \to M$ denote the inclusion. Given a TNBFT $(\sE,\sEb,\cdots)$, the \textbf{canonical submersion} is the composite sheaf map
\begin{equation}
\rho: \sE \to \iota_* \sEb 
\end{equation}
which arises as the composite of 
\begin{enumerate}
\item The restriction
\begin{equation}
\sE(U) \to \sE(U\cap \tubnhd),
\end{equation}
\item the isomorphism
\begin{equation}
\sE(U\cap \tubnhd) \cong (\sEb\hotimes_\beta \Omega^\bullet_{[0,\epsilon)})(U\cap \tubnhd),
\end{equation}
\item
and the ``pullback to $t=0$'' map
\begin{equation}
(\sEb\hotimes_\beta \Omega^\bullet_{[0,\epsilon)})(U\cap \tubnhd) \to \sEb(U\cap \bdyM).
\end{equation}
\end{enumerate}
\end{definition}

\begin{remark}
One can extract a BV-BFV theory (\autocite{CMRquantumgaugetheories}) from a classical TNBFT, using $\sE$ as the bulk space of fields and $\sEb$ as the boundary fields. The assumptions in the definition guarantee that the general procedure of symplectic reduction in \autocite{CMRquantumgaugetheories} outputs $\sEb$ as the phase space for $\sE$.
\end{remark}

\begin{remark}
Occasionally, when we wish to emphasize the interpretation of $\sE$ (respectively, $\sEb$) as formal moduli spaces (as in \autocite{CG2}), we will refer to $\sE$ (respectively, $\sEb$) as a \textbf{local, (--1)-shifted (respectively, 0-shifted) symplectic formal moduli problem}.
\end{remark}

The following facts follow in a straightforward manner from the definitions.

\begin{proposition}
\label{prop: propsofTNBFTs}
\begin{enumerate}
\item The canonical surjection $\rho$, when thought of as a map
\begin{equation}
\sE[-1] \to \iota_* \sEb[-1],
\end{equation}
is a strict map of sheaves of $L_\infty$ algebras on $M$ (i.e. it strictly intertwines the operations $\ell_k$ and $\ell_{k,\partial}$).
\item The only failure of $(\sE_c[-1],\ell_k, \ip)$ to be a precosheaf of cyclic $L_\infty$ algebras on $M$ is encoded in the equation 

\begin{equation}
\label{eq: boundaryterm}
\ip[\ell_1 e_1, e_2]+(-1)^{|e_1|} \ip[e_1,\ell_1 e_2] = \ip[\rho e_1, \rho e_2]_{\partial},
\end{equation} 
where $e_1$ and $e_2$ are compactly-supported sections of $E\to M$ (the support of such sections may intersect $\bdyM$). 
\end{enumerate}
\end{proposition}

\begin{remark}
Given a field theory in the sense of Costello and Gwilliam (and Definition \ref{def: clbv}), one may write down a local functional $S$ on the space of fields, and the higher Jacobi identities for the $L_\infty$ algebra imply that $S$ solves the classical master equation $\{S,S\}=0$.
We avoid that interpretation here, since there are subtleties in defining the Poisson bracket $\{\cdot,\cdot\}$ when~$\bdyM\neq \emptyset$.
\end{remark}

\subsubsection{Examples}
Let us move on to discuss some examples of TNBFTs.

\begin{example}
\label{ex: toplmech}
Let $M=[a,b]$, and $V$ a symplectic vector space.  For $\sE$ we take $\Omega^\bullet_M\otimes V$ with the de Rham differential and the ``wedge-and-integrate'' pairing. It is straightforward to verify that $\sE$ is a TNBFT, with $\sEb= V\oplus V$. This theory is called \textbf{topological mechanics}.
\end{example}

\begin{example}
\label{ex: bf}
Let $M$ be an oriented $n$-manifold with boundary $\bdyM$; let $\fg$ be a finite-rank Lie algebra. \textbf{(Topological) BF theory on $M$} has space of fields
\begin{equation}
\sE = \left(\Omega^\bullet_{M}\otimes \fg[1]\right)\oplus \left(\Omega^\bullet_{M}\otimes \fg^\vee[n-2]\right). 
\end{equation}
The $L_\infty$ structure on $\sE[-1]$ is the natural such structure obtained from considering $\sE[-1]$ as the tensor product of the differential graded commutative algebra $\Omega^\bullet_M$ with the graded Lie algebra $\fg\oplus \fg^\vee[n-3]$.
The pairing $\ip_{loc}$ is given by wedge product of forms, followed by projection onto the top form degree.

The boundary data are given similarly by
\begin{equation}
\sEb = \left(\Omega^\bullet_{\bdyM}\otimes \fg[1]\right)\oplus \left(\Omega^\bullet_{\bdyM} \otimes \fg^\vee[n-2]\right)
\end{equation}
with analogous $L_\infty$ structure and pairing $\ip_{loc, \partial}$.

In general, if $\fg$ is an $L_\infty$ algebra, the same definitions can be made. Furthermore, if $M$ is not orientable, one can make the same definitions by replacing the $\fg^\vee$-valued forms with $\fg^\vee$-valued \emph{twisted} forms.
\end{example}

\begin{example}
\label{ex: mixedbf}
Let $\Sigma$ be a Riemann surface and $\fg$ a Lie algebra. \textbf{Mixed BF theory} is a theory on $\Sigma \times \R_{\geq 0}$ whose space of fields is 
\begin{equation}
\sE = \Omega^{0,\bullet}_{\Sigma}\hotimes_\beta \Omega^\bullet_{\R_{\geq 0}}\otimes \fg[1]\oplus\Omega^{1,\bullet}_{\Sigma}\hotimes_\beta \Omega^\bullet_{\R_{\geq 0}}\otimes \fg^\vee;
\end{equation}
the differential on the space of fields is $\bar \partial+ d_{dR, t}$ (where $t$ denotes the coordinate on $\R_{\geq 0}$); the only non-zero bracket of arity greater than 1 is the two-bracket, which arises from the wedge product of forms and the Lie algebra structure of $\fg$ (and its action on $\fg^\vee$).
One defines $\ip_{loc}$ in an analogous way to the corresponding object in topological BF theory.

The boundary data are given by the space of boundary fields
\begin{equation}
\sEb = \Omega^{0,\bullet}_\Sigma \otimes \fg[1]\oplus \Omega^{1,\bullet}_\Sigma \otimes \fg,
\end{equation}
with similar definitions for the  brackets and pairing.

Mixed BF theory may be similarly formulated on any manifold of the form $X\times N$, where $X$ is a complex manifold and $N$ is a smooth manifold with boundary.
\end{example}

\begin{example}
Let $M$ be an oriented 3-manifold with boundary, and $\fg$ a Lie algebra with a symmetric, non-degenerate pairing $\kappa$. \textbf{Chern-Simons theory} on $M$ is given by space of fields
\begin{equation}
\sE = \Omega^\bullet_M\otimes \fg[1];
\end{equation}
$\sE[-1]$ has the $L_\infty$ structure induced from considering it as the tensor product of the commutative differential graded algebra $\Omega^\bullet_M$ with the Lie algebra $\fg$. The pairing $\ip_{loc}$ uses the wedge product of forms and the invariant pairing $\kappa$. 

The boundary data are encoded in the boundary fields
\begin{equation}
\sEb = \Omega^\bullet_\bdyM \otimes \fg[1],
\end{equation}
with brackets and pairing $\ip_{loc, \partial}$ defined similarly to the analogous structures on the bulk fields $\sE$.

We note that, on manifolds of the form $\Sigma \times \R_{\geq 0}$, Chern-Simons theory is a deformation of mixed BF theory (see \autocite{ACMV} or \autocite{GW} for details). 
\end{example}

\begin{example}
Let $N$ be a manifold without boundary, and let $(\sG, \ell_1, \ell_2)$ be an elliptic differential graded Lie algebra on $N$ (i.e. $\sG$ is a local differential graded Lie algebra such that the complex $(\sG,\ell_1)$ is elliptic).
$\sG$-BF theory is the theory on $N\times \RR_{\geq 0}$ with sheaf of fields
\begin{equation}
\sG\hotimes_\beta \Omega^\bullet_{\RR_{\geq 0}}[1]\oplus \sG^!\hotimes_\beta \Omega^\bullet_{\RR_\geq 0}[-1];
\end{equation}
the brackets and pairing on this sheaf of fields are defined by analogy with the preceding examples.
\end{example}

\begin{remark}
Many of the theories we consider are fully topological in nature (topological BF theory, Chern-Simons theory, the Poisson sigma model, and topological mechanics are all of this form). However, we will see that even for topological theories, we can choose a non-topological boundary condition. This is notably true for Chern-Simons theory.
\end{remark}

\begin{remark}
As noted in the introduction, one of the main sources of examples of TNBFTs is the class of so-called ``degenerate'' BV theories (\autocite{butsonyoo}). 
Every degenerate theory $\cT$ on a manifold $N$ gives rise to a (``non-degenerate'') TNBFT on $N\times \RR_{\geq 0}$, which Butson and Yoo call the ``universal bulk theory for $\cT$''. 
We refer the reader to \autocite{butsonyoo}, Definitions 2.34 and 3.18 for the definitions of these notions, since we do not use them explicitly at any length. 
We note only that BF theory, mixed BF theory, the Poisson sigma model, Chern-Simons theory, and $\sG$-BF theory (on spaces of the form $N\times \R_{\geq 0}$) arise in this way. 
\end{remark}

\subsection{Boundary conditions and homotopy pullbacks}
\label{sec: boundcnds}
In this section, we study boundary conditions for TNBFTs.
Our main goal is to provide a derived geometric interpretation for the imposition of boundary conditions, as discussed in the introduction.
\subsubsection{Boundary conditions}
Equation \eqref{eq: boundaryterm} tells us that the space of fields $\sE$ of a TNBFT on $M$ is \emph{not} $(-1)$-shifted symplectic. However, we do have a map $\rho: \sE(M)\to \sEb(\bdyM)$ and $\sEb(\bdyM)$ has a $0$-shifted symplectic structure. We will construct a Lagrangian structure on the map $\rho$, using Equation \eqref{eq: boundaryterm}. If we are given another Lagrangian $\sL(M)\to \sEb(M)$, then by Example 3.2 of \autocite{calaquedag}, we can expect that the homotopy fiber product 
\[
\sE(M)\times^h_{\sEb(\bdyM)}\sL(M)
\]
have a $(-1)$-shifted symplectic structure. We will call this second Lagrangian $\sL(M)$ a \emph{boundary condition}. The purpose of this section is to carry out this general philosophy more precisely, and in a sheaf-theoretic manner on $M$.

\begin{definition}
\label{def: loclagrangianstr}
Let $\sF$ be the sheaf of sections of a bundle $F\to M$ and suppose that $\sF[-1]$ is a local $L_\infty$-algebra on $M$ with brackets $\ell_i$. Suppose that $\sG$ is a 0-shifted symplectic local formal moduli problem on $\partial M$ (i.e. $\sG$ satisfies the axioms that $\sEb$ does in Definition \ref{def: tnbft}), and $\rho: \sF[-1]\to \iota_*\sG[-1]$ a map of sheaves with the following properties:
\begin{itemize}
\item $\rho$ intertwines the $L_\infty$ brackets (strictly), and
\item $\rho$ is given by the action of a differential operator acting on $\sF$, followed by evaluation at the boundary, followed by a differential operator 
\begin{equation}
\cinfty(\bdyM, F\mid_{\bdyM})\to \sG.
\end{equation}
\end{itemize}

Suppose 
\begin{equation}
h_F: F\otimes F \to \densM.
\end{equation}
is a bundle map. We write $h_{loc}$ for the induced map
\begin{equation}
h_{loc}: \sF\otimes \sF \to \Omega^n_{M,tw}
\end{equation}
and $h$ for the induced pairing on compactly-supported sections of $F$.
\begin{enumerate}
\item  The pairing $h$ is \textbf{a constant local isotropic structure on $\rho$}  precisely if
\begin{equation}
\label{eq: diff}
h(\ell_1 f_1, f_2)+(-1)^{|f_1|}h(f_1, \ell_1 f_2) = \ip[\rho f_1, \rho f_2]_{\sG},
\end{equation}

where $\ip_{\sG}$ is the symplectic pairing on $\sG$, and
\begin{equation}
\label{eq: higherbrackets}
h( \ell_k(f_1,\cdots, f_k), f_{k+1}) \pm h(f_1, \ell_k (f_2, \cdots, f_{k+1}))=0
\end{equation}
for $k>1$.
\item A constant local isotropic structure induces a map of complexes of sheaves
\begin{align}
\Psi:Cone(\rho) &\to \sF_{c}^\vee,\\
\Psi(f,g)(f')&=h(f,f')-\omega(g,\rho(f')).
\end{align}
We say that the isotropic structure is \textbf{Lagrangian} (i.e. that there is a constant local Lagrangian structure on $\rho$) if this map of complexes of sheaves is a quasi-isomorphism. Here, the symbol $\vee$ denotes the sheaf which, to an open $U\subset M$, assigns the strong continuous dual to $\sF_c(U)$. In other words, $\sF^\vee_c$ is the sheaf of distributional sections of $F^!$. 
\end{enumerate}
\end{definition}

\begin{remark}
Let us give a more geometric interpretation of the definition of isotropic structure. To this end, let $\fg$ and $\fh$ be $L_\infty$ algebras and $\rho: \fg\to \fh$ a strict $L_\infty$ map. The algebras $\fg$ and $\fh$ define formal moduli spaces $B\fg$ and $B\fh$. Let us suppose that $B\fh$ is 0-shifted symplectic with symplectic form $\ip_{\fh}$. Then, an isotropic structure on the map $\rho$ is an element $h\in \Omega^2_{cl}(B \fg)$ (the complex of closed two-forms on $B\fg$) such that 
\begin{equation}
\label{eq: isotropicdef}
Q^{TOT} h = \rho^*\left( \ip_{\fh}\right),
\end{equation}
where $Q^{TOT}$ is the total differential on $\Omega^2_{cl}(B\fg)$, which includes a Chevalley-Eilenberg term and a de Rham term. If one requires that both $\ip_{\fh}$, respectively $h$, be constant on $B\fh$, respectively $B\fg$, (i.e. are specified by an element of $\Lambda^2 (\fh[1])^\vee$, respectively $\Lambda^2 (\fg[1])^\vee$), then one obtains precisely Equations \eqref{eq: diff} and \eqref{eq: higherbrackets} from Equation \eqref{eq: isotropicdef}.
\end{remark}

\begin{remark}
Definition \ref{def: loclagrangianstr} is an adaptation to the setting of local $L_\infty$ algebras of the corresponding definitions for derived Artin stacks given in Section 2.2 of \autocite{ptvv}. 
\end{remark}

\begin{lemma}
\label{lem: fieldslagrangian}
The pairing $\ip_{\sE}$ endows the map $\sE\to \sEb$ with a constant local Lagrangian structure.
\end{lemma}

\begin{proof}
Equation \eqref{eq: boundaryterm} implies Equation \eqref{eq: diff}, while the invariance of $\ip_\sE$ under the higher ($k>1$) brackets implies Equation \eqref{eq: higherbrackets}. Hence $\ip_\sE$ gives a constant local isotropic structure on $\rho$. We need only to check that the induced map of sheaves
\begin{equation}
\left( \sE[1]\oplus \iota_* \sEb, \diff+\Qb \pm \rho\right) \to \sE_c^\vee
\end{equation}
is a quasi-isomorphism of sheaves. We prove this as follows: first, we note that for an open $U\subset M$ which does not intersect $\partial M$, we are studying the map
\begin{equation}
\sE[1](U)\to \sE^!(U) \to \sE_c(U)^\vee,
\end{equation}
where the first map is induced from the isomorphism $E[1]\to E^!$ arising from $\ip_{E}$. The composite map is a quasi-isomorphism because the first map is an isomorphism and the second map is the quasi-isomorphism of the Atiyah-Bott lemma (see, e.g., Appendix D of \autocite{CG1}). 

Now, let $U\cong U'\times [0,\delta)$, where $U'$ is an open subset of $\partial M$ and $U\subset \tubnhd$. We will show that both 
\begin{equation}
(\sE(U)[1]\oplus \sEb(U'),Q_{cone})
\end{equation}
and $\sE_c^\vee(U)$
are acyclic. 
Let's start with the latter complex. 
The following heuristic argument explains why one should expect the complex $\sE_c(U)$ to be acyclic.
Because the theory is topological normal to the boundary, any $e\in \sE_c(U)$ is cohomologous to its evaluation at some~$t\in [0,\delta)$.
Because $e$ is compactly-supported, evaluation at a sufficiently large $t$ will give 0. 

To make the preceding argument precise, write $e=e_1+e_2dt$ for $e\in \sE_{c}(U)$, and consider the map 
\begin{align}
&K: \sE_{c}(U)[1]\to \sE_{c}(U)\\
&K(e)(t)=(-1)^{|e_2|+1}\int_t^\delta e_2(s)ds;
\end{align}
one verifies that $K$ is a contracting homotopy for $\sE_{c}(U)$: 
\begin{align}
\diff K(e) &= e_2\wedge dt+(-1)^{|e_2|+1}\int_t^\delta \Qb e_2(s)ds\\
K\diff(e)(t)& = -e_1(\delta)+e_1(t)+(-1)^{|e_2|}\int_t^\delta \Qb e_2(s)ds,
\end{align}
so that $\diff K+K\diff=id$. 
Here, we have used the fact that $e_1(\delta)$ is 0 by the compact-support condition.
By dualizing these data, we construct a contracting homotopy for $\sE_c^\vee$.

Next, consider $(\sE(U)[1]\oplus \sEb(U'), Q_{cone})$. There is a natural map $\phi: \sEb(U')\to \sE(U)\cong \sEb(U')\hotimes_\beta \Omega^\bullet_{[0,\epsilon)}([0,\delta))$ induced from the map $\CC\to \Omega^\bullet_{[0,\epsilon)}([0,\delta))$. Let $C$ denote the mapping cone for the identity map on $\sEb(U')$, and define the map
\begin{align}
\Phi: C&\to \cone(\rho)(U)\\
\Phi(\alpha_0,\alpha_1) &=(\phi(\alpha_0),\alpha_1).
\end{align}
It is straightforward to check that $\Phi$ is a quasi-isomorphism. Hence, $\cone(\rho)(U)$ is acyclic. 
We have shown that, for every point $x\in M$, and any neighborhood $V$ of $x$, we have a neighborhood $U\subset V$ of $x$ on which the sheaf map under study is a quasi-isomorphism. This map is therefore a quasi-isomorphism of complexes of sheaves. 
\end{proof}

We would now like to choose another Lagrangian $\sL\to \sEb$, so that the homotopy pullback $\sL\times^h_{\sEb}\sE$ is $(-1)$-shifted. To make this precise, we need to choose a model category in which to take the homotopy pullback, and we need to introduce an appropriate class of Lagrangians $\sL$ which we will study. The following definition is intended to fulfill the latter aim.

\begin{definition}
\label{def: bdycond}
Let $\sE$ be a TNBFT. A \textbf{local boundary condition} for $\sE$ is a subbundle $L\subset \Eb$ endowed with brackets $\ell_{L,i}$ making $L[-1]$ into a local $L_\infty$ algebra satisfying the following properties:
\begin{itemize}
\item The induced map $\sL\to \sEb$ of sheaves of sections on $\bdyM$ intertwines (strictly) the brackets,
\item $\ip_{\Eb}$ is identically zero on $L\otimes L$, and
\item there exists a vector bundle complement $L'\subset \Eb$ on which $\ip_{\Eb}$ is also zero.
\end{itemize}
Such data are considered a \emph{local} boundary condition because the map $\sL\to \sEb$ arises from a bundle map $L\subset \Eb$. Since we have no need for boundary conditions which are not of this form, we will use the term ``boundary condition'' when we mean ``local boundary condition.''
\end{definition}

\begin{example}
\label{ex: toplmechbc}
Recall from Example \ref{ex: toplmech} that a symplectic vector space $V$ provides a TNBFT on $\R_{\geq 0}$. It is straightforward to check that a Lagrangian subspace $L\subset V=\sEb$ gives a local boundary condition for this theory.
\end{example}

\begin{example}
\label{ex: bfbdycond}
One can define two boundary conditions for BF theory (Example \ref{ex: bf}).
Recall that the space of boundary fields is 
\[
\sEb = \Omega^\bullet_{\bdyM}\otimes \fg [1]\oplus \Omega^\bullet_{\bdyM,tw}\otimes \fg^\vee[n-2];
\]
we may take either of these two summands as a boundary condition. 
We will call the former boundary condition the $A$ condition and the latter the $B$ condition.
\end{example}

\begin{example}
For Chern-Simons theory on an oriented 3-manifold $M$, the sheaf of boundary fields is
\begin{equation}
\Omega^\bullet_{\bdyM}\otimes \fg[1];
\end{equation}
let us choose a complex structure on $\bdyM$. Then, it is straightforward to verify that $\Omega^{1,\bullet}_{\bdyM}\otimes \fg$ gives a local boundary condition, the \textbf{chiral Wess-Zumino-Witten boundary condition}. The Chern-Simons/chiral Wess-Zumino-Witten system for abelian $\fg$ is the central example of \autocite{GRW}.
\end{example}

\subsubsection{The \condfieldsterm fields}
In this section, we define what we mean by a classical bulk-boundary system, and we give the promised homotopy-theoretic interpretation of the imposition of the boundary condition.
\begin{definition}
\label{def: clblkbdysystem}
Given a TNBFT $(\sE,\sEb,\ip)$ and a boundary condition $\sL\subset \sEb$, the \textbf{space of \condfieldsterm fields} $\condfields$ is the complex of sheaves 
\begin{equation}
\condfields(U):= \{e\in \sE(U)\mid \rho(e)\in (\iota_*\sL)(U)\}
\end{equation}
of fields in $\sE$ satisfying the boundary condition specified by $\sL$. In other words, $\condfields$ is the (strict) pullback $\sE\times_{\iota_*\sEb}(\iota_*\sL)$ taken in the category of presheaves of complexes on $M$.

A \textbf{classical bulk-boundary system} is a TNBFT together with a boundary condition.
\end{definition}

\begin{remark}
The term ``bulk-boundary system'' is appropriate for a much more general class of field-theoretic information. However, since we study only this specific type of bulk-boundary system, we omit qualifying adjectives from the terminology.
\end{remark}

\begin{lemma}
\label{ref: tildefieldsBV}
The brackets on $\sE[-1]$ descend to brackets on $\condfields[-1]$, and 
\begin{equation}
(\condfields,\ell_i,\ip_{\sE})
\end{equation}
forms a classical BV theory in the sense of Definition \ref{def: clbv}, except that $\condfields$ is not the sheaf of sections of a vector bundle on $M$.
\end{lemma}
\begin{proof}
The first statement follows from the fact that $\condfields[-1]$ is a pullback of sheaves of $L_\infty$-algebras on $M$. The only thing that remains to be verified is that the the pairing $\ip_{\sE}$ is cyclic with respect to the brackets $\ell_i$, once we restrict to $\condfields$. By our assumptions, the only failure of the brackets to be cyclic for $\sE$ is captured in Equation \eqref{eq: boundaryterm}. Upon restriction to $\condfields$, however, the boundary term in that equation vanishes.
\end{proof}

The previous lemma shows that the pullback 
\begin{equation}\sE\times_{\iota_*\sEb}(\iota_*\sL)\end{equation}
in the category of presheaves of shifted $L_\infty$ algebras on $M$ has what may be interpreted as a $(-1)$-shifted symplectic structure. 
(The infinite-dimensional nature of the space of fields makes this interpretation a bit imprecise, but nevertheless, the space of $\sL$-conditioned fields behaves in much the same way that the space of fields of a BV theory on a manifold without boundary does.) 
In the rest of this section, we explain why this is not an accident. 

As we have noted in Lemma \ref{lem: fieldslagrangian}, the map $\sE\to \iota_* \sEb$ is a Lagrangian map. 
Moreover, the map $\sL\to \sEb$ is also Lagrangian by assumption. 
We should therefore expect the homotopy pullback $\sE\times^h_{\iota_*\sEb}(\iota_*\sL)$ (in an appropriate model category) to have a $(-1)$-shifted symplectic structure.
The space of $\sL$-conditioned fields $\condfields$ is \emph{a priori} only the strict pullback.
However, in Lemma \ref{lem: tildefieldsmodel} below, we will show that $\condfields$ is indeed a model for this homotopy pullback. 

Let us first, however, describe the model category in which we would take the homotopy fiber product. The sheaves $\sE[-1],\iota_*\sEb[-1],\iota_*\sL[-1]$ are presheaves of $L_\infty$-algebras on $M$. In \autocite{hinichsheaves} (Theorem 2.2.1), a model structure on the category of such objects is given. The weak equivalences in this model category are those which induce quasi-isomorphisms of complexes of sheaves after sheafification, and the fibrations $f:M\to N$ are the maps such that $f(U):M(U)\to N(U)$ is surjective (degree by degree) for each open $U$ and such that for any hypercover $V_\bullet\to U$, the corresponding diagram 
\begin{equation}
\begin{tikzcd}
M(U)\ar[r]\ar[d] & \check{C}(V_\bullet,M)\ar[d]\\
N(U)\ar[r]& \check{C}(V_\bullet,N)
\end{tikzcd}
\end{equation}
is a homotopy pullback. This information about the model category is enough to show the following lemma:

\begin{lemma}
\label{lem: tildefieldsmodel}
In the model category briefly described in the preceding paragraph, the sheaf of \condfieldsterm fields $\condfields[-1]$ is a model for the homotopy pullback
 \begin{equation}(
 \sE[-1])\times^h_{\iota_*\sEb[-1]}(\iota_*\sL[-1])
 \end{equation}
 of presheaves of $L_\infty$-algebras on $M$.
\end{lemma}
\begin{proof}
We first note the following: $\sE$ satisfies \v{C}ech descent for arbitrary covers in $M$, and $\sEb,\sL$ do the same on $\bdyM$, by Lemma B.7.6 of \autocite{CG1}. 
Because the \v{C}ech complex for $\iota_*\sL$ (resp. $\iota_*\sEb$) with cover $\{U_\alpha\}$ is identically the \v{C}ech complex for $\sL$ (resp. $\sEb$) with cover $\{U_\alpha\cap \bdyM\}$, $\iota_*\sL$ and $\iota_*\sEb$ satisfy \v{C}ech descent as presheaves on $M$. By Theorem 7.2.3.6 and Proposition 7.2.1.10 of \autocite{HTT}, these presheaves satisfy descent for arbitrary hypercovers. 
(Strictly speaking, the cited results are only proved for presheaves of simplicial sets on $M$; however, using the boundedness of $E$ stipulated in Definition \ref{def: tnbft}, we can shift all objects involved to be concentrated in non-positive degree and then use the Dold-Kan correspondence to show that hyperdescent and descent coincide for presheaves of (globally) bounded complexes.)

We claim that the map $\sE\to \iota_*\sEb$ is a fibration, whence the lemma follows immediately. It is manifest that the maps $\sE(U)\to \sEb(U\cap \bdyM)$ are surjective for every open $U\subset M$. So, it remains to check that the square  
\begin{equation}
\begin{tikzcd}
\sE(U)\ar[r]\ar[d] & \check{C}(V_\bullet,\sE)\ar[d]\\
\sEb(U)\ar[r]& \check{C}(V_\bullet,\sEb)
\end{tikzcd}
\end{equation}
is a homotopy pullback square of complexes for any hypercover $V_\bullet$. This is true because the above square is the outer square of the diagram
\begin{equation}
\begin{tikzcd}
\sE(U) \ar[rrr]\ar[ddd]\ar[rd,"\sim"]&&&\check{C}(V_\bullet,\sE)\ar[ddd]\ar[ld,"\sim"]\\
&\check{C}(V_\bullet,\sE)\ar[r,"\id"]\ar[d]&\check{C}(V_\bullet,\sE)\ar[d]&\\
&\check{C}(V_\bullet,\sEb)\ar[r,"\id"]&\check{C}(V_\bullet,\sEb)&\\
\sEb(U)\ar[rrr]\ar[ur,"\sim"]&&&\check{C}(V_\bullet,\sEb)\ar[lu,"\sim"]
\end{tikzcd};
\end{equation}
all the diagonal maps in the diagram are quasi-isomorphisms, and the inner square is clearly a homotopy pullback square.
\end{proof}

\input{LocalFunctionals}

\section{The Factorization Algebras of Observables}
\label{sec: FAs}
In this section, given a bulk-boundary system $(\sE,\sL)$, we construct a factorization algebra $\Obcl_{\sE,\sL}$ of classical observables for the bulk-boundary system on $M$. $\Obcl_{\sE,\sL}$ will be constructed as ``functions'' on $\condfields$. $\Obcl_{\sE,\sL}$ has the advantage of being easy to define; however, it does not manifestly carry the $P_0$ (shifted Poisson) structure that one expects to find on the space of functions on a (--1)-shifted symplectic space. Hence, we construct also a $P_0$ factorization algebra $\widetilde{\Obcl_{\sE,\sL}}$ and a quasi-isomorphism $\widetilde{\Obcl_{\sE,\sL}}\to \Obcl_{\sE,\sL}$. We closely follow \autocite{CG2}.

\begin{definition}
\label{def: classFA}
Let $(\sE,\sL)$ be a bulk-boundary system. Define 
\begin{equation}
\Obcl_{\sE,\sL}(U):= (\sO(\condfields(U)), d_{CE}) = \left(\prod_{k\geq 0} \underline{CVS}\left(\condfields(U)^{\hotimes_\beta k}, \RR\right)_{S_k}, d_{CE}\right)
\end{equation}
where the tensor product $\hotimes_\beta$ is the completed bornological tensor product of convenient vector spaces. Here, $d_{CE}$ denotes the Chevalley-Eilenberg differential constructed from the structure of $L_\infty$ algebra on $\condfields(U)[-1]$. 
We note that $\Obcl_{\sE,\sL}(U)$ is a differential-graded commutative algebra in the symmetric monoidal category $\CVS$ for all $U$.
Using the functor $\CVS\to \DVS$ and the fact that the tensor product $\hotimes_\beta$ represents the multicategory structure in $\DVS$, we may view $\Obcl_{\sE,\sL}(U)$ as a differential-graded commutative algebra in the multicategory $\DVS$ for all~$U$.

Given disjoint open subsets $U_1,\ldots, U_k\subset M$ all contained in an open subset $V
\subset M$, define the structure map
\begin{equation}
m_{U_1,\cdots, U_k}^V:\Obcl_{\sE,\sL}(U_1)\hotimes_\beta \cdots \hotimes_\beta \Obcl_{\sE,\sL}(U_k)\to \Obcl_{\sE,\sL}(V)
\end{equation}
to be the composite
\begin{align}
\Obcl_{\sE,\sL}(U_1)\hotimes_\beta \cdots \hotimes_\beta \Obcl_{\sE,\sL}(U_k)&\to \Obcl_{\sE,\sL}(V)\hotimes_\beta \cdots \hotimes_\beta \Obcl_{\sE,\sL}(V)\nonumber\\
&\to \Obcl_{\sE,\sL}(V)
\end{align}
where the first map is induced from the restriction maps~$\condfields(V)\to \condfields(U_i)$ and the second map is the induced from the commutative product on~$\Obcl_{\sE,\sL}(V)$.
\end{definition}

\begin{remark}
The construction here is almost identical to that of \autocite{CG2}. We simply impose boundary conditions on the fields of interest.
\end{remark}

\begin{theorem}
\label{thm: classobsformFA}
The differentiable vector spaces $\Obcl_{\sE,\sL}(U)$, together with the structure maps $m_{U_1,\ldots,U_k}^V$, form a factorization algebra.
\end{theorem}

\begin{proof}
The proof is very similar to the analogous proof in \autocite{CG2}. The proof that the classical observables form a prefactorization algebra is identical to that in \autocite{CG2}. For the Weiss cosheaf condition, it suffices, as in \autocite{CG2}, to check the condition for free theories. A similar situation arises in \autocite{GRW}, though there one uses $\Sym(\condfieldscs[1])$ for the classical observables. By the same arguments as in the corresponding proof in \autocite{CG2}, we need only to show that, given any Weiss cover $\fU=\{U_i\}_{i\in I}$ of an open subset $U\subset M$, the map 
\begin{equation}
\label{eq: cechmap}
\bigoplus_{n=0}^\infty \bigoplus_{i_1,\cdots, i_n}\underline{\CVS}\left( \condfields(U_{i_1}\cap\cdots\cap U_{i_n})^{\hotimes_\beta m},\RR\right)_{S_k} [n-1]\to \underline{\CVS} \left(\condfields(U)^{\hotimes_\beta m},\RR\right)_{S_k}
\end{equation}
is a quasi-isomorphism, where the left-hand side is endowed with the \v{C}ech differential and the internal differential induced by $\ell_1$ only.

We show in Section \ref{sec: innerhom} that there is a non-canonical splitting 
\begin{equation}
    \sE(U) \cong \condfields(U) \oplus \sL'(U\cap \bdyM).
\end{equation}
We may therefore understand the mapping spaces appearing in Equation \eqref{eq: cechmap} as quotients of corresponding spaces without boundary conditions imposed (in both categories of interest, $\CVS$ and $\DVS$; this was why it was important to actually produce this splitting).
More precisely, the kernel of the quotient map 
\begin{equation}
\label{eq: quotientoffunctionals}
\underline{\CVS} \left(\sE(U)^{\hotimes_\beta m},\RR\right)_{S_k}\to \underline{\CVS} \left(\condfields(U)^{\hotimes_\beta m},\RR\right)_{S_k}
\end{equation}
is, naturally, the space of functionals which vanish when all inputs satisfy the boundary condition.
The quoted results from the appendix show that one may actually understand the codomain of the above equation as a quotient in $\CVS$ and $\DVS$.
In Lemma A.5.8 of \autocite{CG1}, Costello and Gwilliam describe an explicit contracting homotopy for the total complex in Equation \eqref{eq: cechmap} where no boundary conditions are imposed.
This contracting homotopy involves only addition and (the transpose of) multiplication by compactly-supported smooth functions on $U^m$ constituting a partition of unity for the cover $\{U_i^m\}_{i\in I}$ of $U^m$.
It is obvious that such a contracting homotopy will preserve the kernel of the quotient map Equation \eqref{eq: quotientoffunctionals}; hence it will also furnish a contracting homotopy for the codomain in Equation \eqref{eq: quotientoffunctionals}.
\end{proof}

In \autocite{CG2}, Theorem 6.4.0.1, it is also shown that $\Obcl$ possesses a sub-factorization algebra $\widetilde{\Obcl}$ which has a $P_0$ structure. We have the same situation here. Let us first make a definition:

\begin{definition}
Let 
\[
I\in \underline{\CVS}\left( \condfields(U)^{\hotimes_\beta k}, \RR\right)_{S_k}
\]
with $k\geq 1$. 
Choose a representative $\bar I$ of $I$ in
\[
\underline{\CVS}\left( \condfields(U)^{\hotimes_\beta k}, \RR\right).
\]
$\bar I$ determines $k$ maps 
\begin{equation}
\underline{\CVS}\left( \condfields(U)^{\hotimes_\beta (k-1)}, \RR\right)_{S_{k-1}}\to \underline{\CVS}\left( \condfields(U), \RR\right)
\end{equation}
by the hom-tensor adjunction in $\CVS$; the $k$ maps correspond to the $k$ choices of tensor factor for which to perform the hom-tensor adjunction.
This collection of $k$ maps does not depend on the choice of $\bar I$.
We say that $I$ has \textbf{smooth first derivative} if all $k$ maps have image in $\condfieldscs[1](U)$. 
We consider this condition to be vacuously satisfied when $k=0$. Given $J\in \sO(\condfields(U))$, we say that $J$ has smooth first derivative if each of its Taylor components does.
\end{definition} 

\begin{theorem}
\label{thm: P0classobsFA}
Let
$\widetilde{\Obcl_{\sE,\sL}}(U)$ denote the subspace of $\Obcl_{\sE,\sL}(U)$ consisting of functionals with smooth first derivative. 
\begin{enumerate}
\item $\widetilde{\Obcl_{\sE,\sL}}$ is a sub-factorization algebra of $\Obcl_{\sE,\sL}$. 
\item $\widetilde{\Obcl_{\sE,\sL}}$ posseses a Poisson bracket of degree $+1$.
\item The inclusion $\widetilde{\Obcl_{\sE,\sL}}\to \Obcl_{\sE,\sL}$ is a quasi-isomorphism.
\end{enumerate}
\end{theorem}

\begin{proof}
Again, the proof follows \autocite{CG2}.
The proof that $\widetilde{\Obcl_{\sE,\sL}}$ is a sub-prefactorization algebra of $\Obcl_{\sE,\sL}$ is identical to the one in \autocite{CG2}. One point requires comment, however: $\widetilde{\Obcl}$ is closed under the Chevalley-Eilenberg differential on $\Obcl_{\sE,\sL}$ because $(\condfieldscs[-1],\ell_l,\ip)$ is a precosheaf of cyclic $L_\infty$ algebras. Indeed, any functional with smooth first derivative is a sum of functionals of the form
\begin{equation}
I(e_1,\ldots,e_k) = \ip[D(e_1,\cdots, e_{k-1}), e_k],
\end{equation}  
where $D: \condfields^{\hotimes_\beta (k-1)} \to \condfieldscs[1]$ is a continuous map. One can check directly that applying the Chevalley-Eilenberg differential to such a functional gives another such functional.
It is important that $(\condfieldscs[-1],\ell_1,\ip)$ is a precosheaf of \emph{cyclic} $L_\infty$ algebras because this allows one to ``integrate by parts,'' i.e. use the equality
\begin{equation}
\ip[D(e_1,\ldots,e_{k-1}),\ell_1 e_k]=\pm\ip[\ell_1 D(e_1,\ldots, e_{k-1}),e_k]
\end{equation}
and its analogues for the higher brackets $\ell_2,\ell_3, \ldots$.

The construction of the shifted Poisson bracket is also identical to the construction in \autocite{CG2}. The Weiss cosheaf condition will be satisfied once we prove that $\widetilde{\Obcl_{\sE,\sL}}\to \Obcl_{\sE,\sL}$ is a quasi-isomorphism. Hence, statement (3) of the theorem is the only one that remains to be proved.

The essential ingredient in the proof of statement (3) is the fact that the inclusion 
\begin{equation}
\condfieldscs(U)[1]\to \condfields^\vee(U)
\end{equation}
is a quasi-isomorphism with a homotopy inverse for certain open subsets $U$, cf. Proposition A.1 of \autocite{GRW}.

Just as in \autocite{CG2}, we may assume that the theory is free. We let 
\begin{equation}
\Sigma^k: \innerhom{\condfields(U)^{\hotimes_\beta k}}{\RR}\to \innerhomsym{\condfields(U)^{\hotimes_\beta k}}{\RR}{k}
\end{equation}
denote the symmetrization map. We let $\Gamma_n$ denote $(\Sigma^{k})^{-1}\widetilde{\Obcl_{\sE,\sL}}(U)$. 
For each $j=1,\ldots, k$ Let $\Upsilon_j$ denote the following map:
\begin{align}
    \Upsilon_j &: \innerhom{\condfields(U)^{\hotimes_\beta(k-1)}}{\condfieldscs(U)[1]}\to \innerhom{\condfields(U)^{\hotimes_\beta k}}{\RR}\\
    \Upsilon_j& (J) = \ip[J(e_1,\ldots, e_{j},e_{j+1},\ldots, e_k),e_j].
\end{align}
We will abuse notation and use $\Upsilon_j$ to denote also the image of the map described in the above equation.
Just as in \autocite{CG2}, we can identify
\begin{equation}
\Gamma_k = \bigcap_{j=1}^{k} \Upsilon_j.
\end{equation}
It suffices to show that the inclusion 
\begin{equation}
\label{eq: incl}
\Gamma_k \hookrightarrow \innerhom{\condfields(U)^{\hotimes_\beta k}}{\RR}
\end{equation}
is an equivalence, since symmetrization is an exact functor. More generally, let $\{U_i\}_{i=1}^k$ be open subsets of $M$, and define $\Gamma_{k,\{U_i\}}$ by a similar intersection.

We will show that the natural inclusion 
\begin{equation}
\label{eq: incltwo}
\Gamma_{k,\{U_i\}}\hookrightarrow \innerhom{\bigotimes_{i=1}^k\condfields(U_i)}{\RR}
\end{equation}
is a quasi-isomorphism possessing a homotopy inverse when each $U_i$ is either contained entirely in $M\backslash \bdyM$ or is of the form $U'_i\times [0,\delta_i)$ for $U'_i$ open in $\bdyM$ (and using our fixed tubular neighborhood of $\bdyM$). We will call such $U_i$ ``somewhat nice.'' Let us explain why this proves that the inclusion in Equation \eqref{eq: incl} is a quasi-isomorphism. It follows from  the results of Appendix \ref{sec: innerhom} that 
\begin{equation}
\innerhom{\bigotimes_{i=1,i\neq j}^{i=k} \condfields(U_i)} {\condfieldscs(U_j)}
\end{equation}
is a quotient of 
\begin{equation}
\innerhom{\bigotimes_{i=1,i\neq j}^{i=k} \sE(U_i)} {\condfieldscs(U_j)}
\end{equation}
Now, let $V=U^n$. We may cover $V$ by products of somewhat nice sets in $M$. Let $\fV=\{U_i\}_{i\in I}$ be such an open cover. By taking the dual statement to that of Lemma A.5.7 of \autocite{CG1}, we find a contracting homotopy for the mapping cone of the map
\begin{equation}
\check{C}(\fV, (\sE^\vee)^{\hotimes_\beta k})\to (\sE^\vee(V))^{\hotimes_\beta k}.
\end{equation} 
This contracting homotopy involves only multiplication by smooth, compactly-supported functions on $M^n=M\times \cdots \times M$. 
We have seen that this cochain homotopy descends to one for the map
\begin{equation}
\check{C}\left(\fV, \innerhom{\condfields^{\hotimes_\beta k}(\cdot)}{\RR} \right)\to \innerhom{\condfields(V)^{\hotimes_\beta k}}{\RR}.
\end{equation} 
Using the explicit characterization of the spaces in the intersection defining $\Gamma_{k,\{U_i\}}$, the cochain homotopy also descends to one for $\Gamma_{k,\{U_i\}}$.

 Hence, we have a commuting diagram
\begin{equation}
\begin{tikzcd}
\check{C}(\fV, \Gamma_{k,\cdot}) \ar[r,"\sim"]\ar[d]& \Gamma_{k}\ar[d]\\
\check{C}\left(\fV, \innerhom{\condfields^{\hotimes_\beta k}(\cdot)}{\RR}\right) \ar[r,"\sim"]& \innerhom{\condfields^{\hotimes_\beta k}(V)}{\RR}
\end{tikzcd}
\end{equation}
where the top and bottom maps are quasi-isomorphisms. 
Here, we abuse notation slightly and let $\Gamma_{n,V'}$ denote $\Gamma_{n,\{U_i\}}$ when $V'=U_1\times \cdots U_n$. 
We are interested in showing that the right-hand arrow in the diagram is a quasi-isomorphism. 
The finite intersection of any number of products of somewhat nice subsets is again a product of somewhat nice sets. 
Therefore, if the map of Equation \eqref{eq: incltwo} is a quasi-isomorphism for $U$ somewhat nice, the left-hand map in the above commuting diagram will be a quasi-isomorphism. 
It follows that the map of \eqref{eq: incl} will be a quasi-isomorphism, since that map is also the right-hand map in the commuting diagram.

Let us now proceed to show that the map of Equation \eqref{eq: incl} is a quasi-isomorphism. To prove this, it suffices---just as in the corresponding proof in \autocite{CG2}---to show that the map
\begin{equation}
\condfieldscs(U_j)[1]\to \innerhom{\condfields(U_j)}{\RR}
\end{equation}
is a quasi-isomorphism when $U\cap \bdyM=\emptyset$ or when $U\cong U' \times [0,\delta)$. For $U\cap \bdyM=\emptyset$, this is the Atiyah-Bott lemma (Appendix D of \autocite{CG1}). For $U\cong U'\times [0,\delta)$, this is shown in Proposition A.1 of \autocite{GRW}
\end{proof}

\section{Examples}
\label{sec: classexamples}
In this section, we study three examples. The quantization of the first appears in \autocite{GRW}, the quantization of the second appears in Chapter 5 of \autocite{ERthesis}, and the quantization of the third is the subject of present work. 

\subsection{Topological Mechanics}
The goal of this section is to study the factorization algebra for topological mechanics. 
Recall from Examples \ref{ex: toplmech} and \ref{ex: toplmechbc} that a symplectic vector space $V$ and a Lagrangian $L\subset V$ define a free bulk-boundary system on $\R_{\geq 0}$ known as topological mechanics. 
The procedure of the previous section constructs a factorization algebra $\Obcl_{V,L}$ on $\R_{\geq 0}$ for these choices. Our goal is to study this factorization algebra.

Given an associative algebra $A$ and a right $A$-module $M$, there is a factorization algebra $\cF_{A,M}$ on $\R_{\geq 0}$ which assigns $A$ to any open interval, and $M$ to any interval containing $0$ (see \S 3.3.1 of \autocite{CG1} for details). 
The structure maps are determined by the multiplication in $A$ and the right-module action of $A$ on $M$. 
We will see that the cohomology factorization algebra of topological mechanics is isomorphic to one of the form $\cF_{A,M}$, for appropriate $A$ and $M$. 

Let $\sO(V)=\Sym(V^\vee)$ denote the symmetric algebra of polynomial functions on $V$, and similarly for $\sO(L)$. The inclusion $L\to V$ induces a restriction of functions map $\sO(V)\to \sO(L)$ which defines a right $\sO(V)$-module structure on $\sO(L)$. 

We would like to say that $\Obcl_{V,L}$ is equivalent to $\cF_{\sO(V),\sO(L)}$; however, $\Obcl_{\sE,\sL}$ is defined in terms of a space of power series on $\condfields$, while $\sO(V)$ and $\sO(L)$ are polynomial algebras. 
To remedy this, one may also consider, for each $U$, the space $\Obcl_{V,L,poly}(U)$ consisting only of polynomial functions on $\condfields(U)$. It is easy to verify that $\Obcl_{V,L,poly}$ forms a sub-factorization algebra of $\Obcl_{V,L}$.

\begin{lemma}
The factorization algebra $\Obcl_{V,L,poly}$ is quasi-isomorphic to the factorization algebra $\cF_{\sO(V),\sO(L)}$. 
\end{lemma} 

\begin{proof}
Let us choose a Lagrangian complement $L'$ to $L$ in $V$.
The sheaf of $\sL$-conditioned fields is
\begin{equation}
 \condfields(U) = \Omega^\bullet_M(U)\otimes L\oplus \Omega^\bullet_{M,D}(U)\otimes L'.
\end{equation}
Here, $\Omega^\bullet_{M,D}(U)$ is $\Omega^\bullet_M(U)$ if $U$ does not contain $t=0$ and otherwise is the space of de Rham forms on $U$ whose pullback to $t=0$ vanishes.
Let $\cS_{pre}$ be the presheaf of vector spaces on $M$ which assigns $V$ to any open set not containing 0 and $L$ to any open set which does contain 0.
Let $\cS$ denote the sheafification of $\cS_{pre}$.
There is a natural map of sheaves $\cS\to \condfields$ and it is straightforward to show that this map is a quasi-isomorphism, for example by extending this map to a homotopy equivalence on connected intervals.
Hence, there is also a quasi-isomorphism
\begin{equation}
    \Obcl \to \sO(\cS).
\end{equation}
One quickly verifies that $\sO(\cS)=\cF_{\sO(V),\sO(L)}$.
This completes the proof.
\end{proof}

\subsection{BF Theory in One Dimension}

In this section, we study the factorization algebra of observables of one-dimensional BF theory with $B$ boundary condition (cf. Examples \ref{ex: bf} and \ref{ex: bfbdycond}).

Namely, we show:

\begin{proposition}
\label{prop: 1dbfclass}
Let $A$ denote the algebra
\[
C^\bullet(\fg, \Sym(\fg[1]))
\]
and 
$M$ the right $A$-module
\[
\Sym(\fg[1]).
\]
There is a quasi-isomorphism of factorization algebras
\begin{equation}
    \Upsilon: \Obcl_{\sE,\sL}\to \cF_{A,M}
\end{equation}
on $\RR_{\geq 0}$, where $\Obcl_{\sE,\sL}$ is the factorization algebra of observables for BF theory with $B$ boundary condition.
\end{proposition}

\begin{proof}
Let $\cS$ denote the sheafification of the following presheaf on $\RR_{\geq 0}$:
\begin{equation}
    \cS_{pre}(U) = \left\{
    \begin{array}{lr}
    \fg\ltimes \fg^\vee[-2]     & U\cap \{0\}=\emptyset  \\
    \fg^\vee[-2] & U \cap \{0\} = \{0\}     
    \end{array}
    \right.
\end{equation}
As we have written it, $\cS$ is naturally a sheaf of graded Lie algebras on $\RR_{\geq 0}$.
The map 
\begin{equation}
\label{eq: 1dbfclassquasiisooffields}
    \cS \to \condfields[-1]
\end{equation}
which includes elements of $\fg\ltimes \fg^\vee[-1]$ as constant functions is a map of sheaves of dg Lie algebras; in fact, it is straightforward to show that it is a quasi-isomorphism.
An inverse quasi-isomorphism
\begin{equation}
    \condfields(U)[-1]\to \cS(U),
\end{equation}
for $U$ a connected open subset, is obtained by pulling back along the inclusion of a point $t\in U$ into $U$.
(This inverse does not define a map of sheaves of dg Lie algebras, however. The inverse map is simply used to verify that the original map is, open-by-open on $\RR_{\geq 0}$, a quasi-isomorphism.)
Let $C^\bullet(\cS)$ be the factorization algebra on $\RR_{\geq 0}$ which, to an open $U\subset \RR_{\geq 0}$, assigns the Chevalley-Eilenberg cochain complex of $\cS(U)$ (cf. Definition 3.6.1 of \autocite{CG1}, though the construction here is slightly different).
The map of Equation \ref{eq: 1dbfclassquasiisooffields} induces a quasi-isomorphism
\begin{equation}
    \Obcl_{\sE,\sL}\to C^\bullet(\cS);
\end{equation}
it is straightforward to verify that~$C^\bullet(\cS)\cong \cF_{A,M}$.
This completes the proof.
\end{proof}

\subsection{BF Theory in Higher Dimensions}
We may also extend the discussion for BF theory above to $\HH^{n+1}$ for positive $n$.
In that case, let $\pi: \HH^{n+1}\to \RR^n$ be the projection onto the boundary, and let $\mathring{\HH}^{n+1}$ denote the interior of the upper half-space.
We will describe $\pi_*\left( \Obcl\right)$ and $\pi_*\left(\left.\Obcl\right|_{\mathring{\HH}^{n+1}}\right)$ for BF theory with B boundary condition.
To this end, we need candidate factorization algebras on $\RR^n$.
That is the purpose of the following definition:

\begin{definition}
Let $\cF^{cl}_\del$ denote the factorization algebra on $\RR^n$ which assigns, to an open subset $U$, the cochain complex
\[
\cF^{cl}_\del(U):=\left(\Sym\left(\Omega^\bullet_{c,\RR^n}(U)\otimes \fg[1]\right),d_{dR}\right),
\]
with factorization structure maps arising in an analogous way to those for the classical observables of a bulk-boundary system.
Note that $\Omega^\bullet_{c,\RR^n}$---and by extension $\cF^{cl}_\del$---is naturally a pre-cosheaf of dg-modules for the sheaf of Lie algebras $\Omega^\bullet_{\RR^n}\otimes \fg$.
Therefore, we let $\cF^{cl}_{blk-bdy}$ denote the factorization algebra on $\RR^n$ which assigns, to an open subset $U$, the cochain complex
\[
\cF^{cl}_{blk-bdy}(U):=C^\bullet(\Omega^\bullet_{\RR^n}(U)\otimes \fg, \cF^{cl}_\del(U)),
\]
with structure maps induced using the natural commutative product on these chain complexes.
\end{definition}

Now, we can state and prove the main result of this subsection:

\begin{lemma}
Let $(\sE,\sL)$ denote the bulk-boundary system consisting of BF theory on $\HH^{n+1}$ with $B$ boundary condition. There are equivalences of factorization algebras
\begin{align*}
\cF^{cl}_{\del} &\to \pi_*(\Obcl_{\sE,\sL}),\\
\cF^{cl}_{blk-bdy}&\to \pi_*\left(\left. \Obcl_{\sE,\sL}\right|_{\mathring{\HH}^{n+1}}\right)
\end{align*}
on $\RR^n$.
\end{lemma}

\begin{remark}
Before we embark on the proof of the Lemma, let us remark on a number of natural questions that remain unanswered here.
First, we note that $\cF^{cl}_\del$ has a natural $P_0$ structure which is given on the linear elements $\Omega^\bullet_{c,\RR^n}\otimes \fg[1]$ by the natural Lie bracket on this space of linear elements.
On the other hand, we have seen that $\widetilde \Obcl_{\sE,\sL}$, which is equivalent to $\Obcl_{\sE,\sL}$ as a factorization algebra, also possesses a $P_0$ structure.
One may ask whether, in a suitable sense, the first quasi-isomorphism of the lemma is an equivalence of $P_0$ factorization algebras.
We expect an affirmative answer to this question.
However, to articulate the ``suitable sense'' in which the map of the lemma is a map of $P_0$ prefactorization algebras requires more technology that we have developed here.
Namely, we expect the map of the lemma to give rise to an isomorphism in the homotopy category of prefactorization algebras valued in $P_0$ algebras.
To make the desired statement, we would need to use machinery from the homotopy theory of algebras over colored operads \autocite{bergermoerdijk}.

In the terminology of Butson and Yoo, BF theory on $\HH^{n+1}$ is the ``universal bulk theory'' for the ``degenerate'' theory whose classical observables are $\cF^{cl}_\del$.
(The terminology stems from the---to our knowledge---conjectural status of the observables $\pi_*\left(\left.\Obcl_{sE,\sL}\right|_{\mathring{H}^{n+1}}\right)$ as the open-by-open $P_0$ center of $\cF^{cl}_\del$ \autocite{safronovlectures}.)
An equivalence of $P_0$ structures between $\cF^{cl}_\del$ and $\pi_*\left(\Obcl_{\sE,\sL}\right)$ would lend more weight to the term ``universal bulk theory.''

The second natural question to ask is the following: it is easy to show that $\cF^{cl}_\del(\RR^n)\simeq \Sym(\fg[1-n])$.
As we have remarked, $\cF^{cl}_\del$ is a locally-constant prefactorization algebra on $\RR^n$ in the category of $P_0$ algebras. 
By a result of Lurie \autocite{higheralgebra}, we may therefore think of $\cF^{cl}_\del$ as an $E_n$-algebra in the category of $P_0$ algebras.
Furthermore, Safronov \autocite{safronovadditivity} has shown that the infinity-category of $E_n$-algebras in the category of $P_0$ algebras is equivalent to the infinity category of $P_n$-algebras.
Hence, $\cF^{cl}_\del$ may be thought of as a $P_n$-algebra.
On the other hand, $\Sym(\fg[1-n])$ has a natural $P_n$ structure, whose Poisson bracket is given on the linear elements by the Lie bracket on $\fg$.
It is natural to ask whether these two $P_n$ structures on equivalent chain complexes coincide.
We again expect an affirmative answer, though this has not been explicitly shown.
\end{remark}

\begin{proof}[Proof of Lemma]
For any open subset of $\HH^{n+1}$ of the form $U\times \RR_{\geq 0}$, we note that we may write the space of fields for BF theory in the form
\begin{align*}
\condfields(U&\times \RR_{\geq 0})\cong \\
& \Omega^\bullet_{\RR^n}(U)\hotimes_\beta \Omega^\bullet_{\RR_{\geq 0},D}(\RR_{\geq 0})\hotimes_\beta \fg[1]\oplus \Omega^\bullet_{\RR^n}(U)\hotimes_\beta \Omega^\bullet_{\RR_{\geq 0}}(\RR_{\geq 0})\hotimes_\beta \fg^*[n-2];
\end{align*}
by applying the techniques of the previous examples in the normal direction, we obtain natural pairs of quasi-isomorphisms, one for each open subset $U\subset \RR^n$,
\[
\condfields(U\times \RR_{\geq 0}) \leftrightarrow \Omega^\bullet_{\RR^n}(U)\otimes \fg^*[n-2].
\]
The rightward-pointing arrow is pullback to the boundary on the $B$ fields and $0$ on the $A$ fields; the leftward-pointing arrow is the pullback along $\pi$.
Combining this observation with the Atiyah-Bott equivalence
\[
\Omega^\bullet_{\RR^n, c}\otimes \fg[1] \to \underline{\CVS}(\Omega^\bullet_{\RR^n}\otimes \fg^*[n-2],\RR),
\]
we obtain the first quasi-isomorphism of the lemma.

The second quasi-isomorphism is proved similarly, noting that the Atiyah-Bott quasi-isomorphism is an equivalence of $\Omega^\bullet_{\RR^n}(U)\otimes \fg$ representations.
\end{proof}

%% file: LocalFunctionals.tex
\section{Local Action Functionals and $D$-modules}
\label{sec: localfcnls}
In this section, we use the language of $D$-modules to describe the local action functionals in a TNBFT. Our approach extends the discussion of Section 6 of Chapter 5 of \cite{cost}. 
In so doing, we step somewhat outside the main line of development of this chapter.
The results of this section are not needed in the construction of a factorization algebra of classical observables, which we discuss in the next section; however, we present them here because they concern the classical theory and will be relevant in the obstruction theory for the quantization of a given classical theory.

Let us recall the double purpose of the space $\sO_{loc}(\sE)$ of local functionals associated to a field theory $\sE$ on a manifold $M$ without boundary.
Before choosing the brackets on $\sE[-1]$, or equivalently the action $S$, (but after choosing the pairing $\ip_{loc}$), $\sO_{loc}(\sE)[-1]$ is a graded Lie algebra; the Lie bracket $\{\cdot,\cdot\}$ on $\sO_{loc}(\sE)[-1]$ is the Poisson bracket induced from the pairing $\ip_{loc}$ on $E$.
A classical field theory on $M$ with underlying graded space of fields $\sE$ and pairing $\ip_{loc}$ is then given by a Maurer-Cartan element $S\in \sO_{loc}(\sE)[-1]$.
(Usually, we require $S$ to have no linear or constant terms. We may restrict the functionals in $\sO_{loc}(\sE)[-1]$ suitably to exclude this possibility.)
The Maurer-Cartan equation $\{S,S\}=0$ is simply the classical master equation in other terms.
Given such a Maurer-Cartan element $S$, $\sO_{loc}(\sE)[-1]$ acquires a differential $\{S,\cdot\}$; the triple $(\sO_{loc}(\sE)[-1],\{S,\cdot\}, \{\cdot, \cdot\})$ determines a differential-graded Lie algebra.
With suitable assumptions on the degree of dependence of the functionals on the fields, this differential-graded Lie algebra governs the deformations of the classical theory determined by $S$.
This is one interpretation of the local action functionals.

The other interpretation of the complex $(\sO_{loc}(\sE),\{S,\cdot\})$ is as the obstruction-deformation complex of the classical field theory.
Costello shows (see Chapter 5 of \autocite{cost}) that, given a quantum field theory defined modulo $\hbar^n$, the obstruction to lifting it to a theory defined modulo $\hbar^{n+1}$ is determined by a degree +1 cohomology class in the obstruction-deformation complex.
When this obstruction vanishes, the degree 0 cohomology classes describe the space of inequivalent lifts of the theory to order $\hbar^{n+1}$.
This is consistent with the interpretation of the degree 0 cohomology classes in $(\sO_{loc}(\sE),\{S,\cdot\})$ as the gauge-equivalence classes of first-order deformations of the classical theory (the Maurer-Cartan elements in the differential-graded Lie algebra $\sO_{loc}(\sE)[-1]\otimes \RR[\epsilon]/(\epsilon^2)$).

In practice, there are a number of computational tools available to study the cohomology of the obstruction-deformation complex; there are fewer available tools to find a small model for the Lie algebra structure.
This is particularly pronounced once one $M$ possesses a boundary.
In this case, the most natural generalization of $\sO_{loc}(\sE)$, which we denote by the symbol $\Oloc$, no longer possesses a Lie bracket.
Nevertheless, we endow $\Oloc$ with a differential analogous to $\{S,\cdot\}$.
The complex $\Oloc$ will have an interpretation as the obstruction-deformation complex for quantizations of the classical field theory (cf. \autocite{ERthesis}).
Moreover, the tricks that one uses in the computation of the cohomology of this complex when $\bdyM= \emptyset$ apply equally well when $\bdyM \neq \emptyset$.
We will apply some of these tools to BF theory in Section \ref{subsec: bfthyloc}.

We will leave the construction of a Lie bracket on $\sO_{loc}(\condfields)$ to future work. We remark on one possible approach to the construction of such a bracket.
Below (see Section \ref{sec: FAs}), we construct a $P_0$ bracket on the observables which have \emph{smooth first derivative}.
When $\bdyM=\emptyset$, local functionals have smooth first derivative, and the usual Poisson bracket of local functionals coincides with the $P_0$ bracket on observables obtained in this way.
When $\bdyM\neq \emptyset$, not all local functionals have smooth first derivative. 
Using the $P_0$ bracket defined in Section \ref{sec: FAs}, we may therefore form the Poisson bracket of local functionals $I$ and $J$ if both $I$ and $J$ have smooth first derivative.
Unfortunately, the obstruction theory for quantum bulk-boundary systems does not necessarily produce functionals with smooth first derivative.
We leave comparison of these two types of functionals for future work.

\subsection{Local action functionals}

When $\partial M=\emptyset$, one defines local action functionals to be of the form 
\begin{equation}
\varphi \mapsto \int_M D_1 \varphi\cdots D_k \varphi d\mu,
\end{equation}
where the $D_i$ are differential operators $\sE\to \cinfty_M$ and $d\mu$ is a density on $M$. We will define these to be local action functionals in the case $\bdyM\neq \emptyset$ as well, but one finds that certain local functionals in this sense (if they are total derivatives) can be written as integrals of a similar nature, \emph{but over $\partial M$}. Hence, the distinction between a local action functional and a codimension-one operator becomes blurry.

Before we define local functionals, we make the following definition, which will be useful in the sequel---in this chapter and especially in the quantum story (Chapter 4 of \autocite{ERthesis}).

\begin{definition}
Let $V$ be a locally convex topological vector space.
The space \textbf{power series} on $V$ is the following convenient vector space
\begin{equation}
    \sO(V):= \prod_{k=0}^\infty \underline{CVS}(V^{\hotimes_\beta k},\RR),
\end{equation}
where the subscript ``bdd'' means we take only the bounded linear maps of topological vector spaces, and the tensor product $\hotimes_\beta$ is the completed bornological tensor product (Section B.5.2 of \autocite{CG1}). 
\end{definition}

We may therefore define $\sO(\sE), \sO(\sE_c), \sO(\condfields),$ and $\sO(\condfieldscs)$.

\begin{definition}
\label{def: localfcnls}
The \textbf{space of local action functionals} for the field theory $\sE$ with boundary condition $\sL$ is the (linear) subspace of $\sO(\condfieldscs)$
spanned by functionals of the form
\begin{equation}
\label{eq: localfcnls}
I(\varphi)= \int_M D_1\varphi \cdots D_j\varphi d\mu.
\end{equation}
Here, $\condfieldscs$ denotes the subspace of those fields which have compact support and satisfy the required boundary condition; all $D_i$ are differential operators on $\sE$, and in particular, may depend on the normal derivatives of the field $\varphi$ near the boundary. 
\end{definition}

\begin{notation}
\label{not: lclfcnls}
Let us establish the following notation:
\begin{itemize}
\item We denote the space of local functionals by $\Oloc$.
\item We denote by the symbol $\Olocred$ the quotient of $\Oloc$ by the space of constant functionals.
\item We denote by the symbol $\sO_{loc}(\sE)$ the space of functionals in $\sO(\sE_c)$ of the form \eqref{eq: localfcnls}.
\item We denote by the symbol $\sO_{loc,red}(\sE)$ the quotient of $\sO_{loc}(\sE)$ by the space of constant functionals.
\end{itemize}
\end{notation}
\begin{remark}
\label{rem: loclfcnlsrmk}
\begin{enumerate}
\item Though $\Oloc$ is a subspace of $\sO(\condfieldscs)$ (and this latter space is endowed with a topology), we do not wish to view $\Oloc$ as a topological vector space. We simply remember the vector space structure underlying $\Oloc$.
\item Note that the same functional can be written in two different ways, i.e. the representation of a local functional $I(\phi)$ in the form $\int_M D_1\varphi \cdots D_j \varphi d\mu$ is not unique.  
\end{enumerate}
\end{remark}

We note that $\Oloc$ contains functionals of the form 

\begin{equation}
I(\varphi)=\int_{\partial M} D_{1}(\varphi)\mid_{\partial M}\cdots D_{k}(\varphi)\mid_{\partial M}\mathrm{d} \nu = -\int_M \frac{\del}{\del t}\left( f_0(t) D_1\phi \cdots D_k\phi\right) \mathrm{d} \nu \,\mathrm{d} t,
\end{equation}
where the $D_i$ are differential operators on $\sE$ and $\mathrm{d}\nu$ is a density on $\bdyM$. 
Here, $f_0$ is a compactly supported function on $[0,\epsilon)$ which is $1$ at $t=0$.
Writing $I$ as on the right-hand side of the above equation makes clear that it is local.

We obtain a natural filtration of $\Oloc$ with $F^0\Oloc$ the space of such functionals and $F^1\Oloc = \Oloc$. This filtration will be useful to us in the sequel.

The Chevalley-Eilenberg differential on $\sO(\condfieldscs)$ preserves $\Oloc$, so that $\Oloc$ is a cochain complex. Moreover, $\Oloc$ evidently has the structure of a presheaf on $M$, since $\condfieldscs$ is a cosheaf. Thus, $\Oloc$ is a complex of presheaves on $M$. It is also easily seen to be a complex of sheaves on $M$.

\subsection{The case $\bdyM=\emptyset$}
\label{sec: localfcnlsnobdy}
Let us first briefly review the relationship between $D$-modules and local action functionals in the case $\bdyM=\emptyset$. Our discussion follows Section 6 of Chapter 5 of \cite{cost}.

\begin{notation}
Given a vector bundle $E$ over a manifold $M$, we let $J(\sE)$ denote the (infinite-rank) bundle over $M$ whose fiber at a point $x\in M$ is the space of formal germs of sections of $E$ at the point $x$. We let $\sJ(\sE)$ denote the sheaf of sections of $J(\sE)$. $\sJ(\sE)$ has a natural topology which it obtains as the inverse limit of the finite-rank bundles of $r$-jets (as $r\to \infty$). We let $\sJ(\sE)^\vee$ denote the $\cinfty_M$ module
\begin{equation}
\Hom_{\cinfty_M}(\sJ(\sE),\cinfty_M),
\end{equation}
where $\Hom_{\cinfty_M}$ denotes $\cinfty_M$-module homomorphisms which are continuous for the inverse limit topology on $\sJ(\sE)$.

We let 
\begin{equation}
\sO(\sJ(\sE))=\prod_{k\geq 0}\Sym^k_{\cinfty_M} \sJ(\sE)^\vee.
\end{equation} 
The symbol $\sO_{red}(\sJ(\sE))$ denotes a similar product over $k>0$.
\end{notation}
$J(\sE)$ carries a canonical connection and so $\sJ(\sE)$ carries the structure of a left $D_M$-module. This $D_M$-module structure gives rise to one on $\sJ(\sE)^\vee$ and $\sO(\sJ(\sE))$. The $L_\infty$ brackets on $\sO(\sE)$, because they are polydifferential operators, can be used to endow $\sO(\sJ(\sE))$ with a Chevalley-Eilenberg differential which respects this $D_M$-module structure. Let us recall also that the bundle of densities $\Dens_M$ carries a structure of right $D_M$-module, where vector fields act on densities by the formal adjoints of their actions on functions (equivalently, by Lie derivatives).
When $\bdyM\neq \emptyset$, we will see that these two $D_M$-module structures are no longer the same; in fact, one of these $D_M$-module structures will cease to be defined.

The following is Lemma 6.6.1 of \cite{cost}:
\begin{lemma}
Suppose $\bdyM=\emptyset$. There is a canonical map of sheaves
\begin{equation}
\Dens_M\otimes_{D_M}\sO(\sJ(\sE))\to \sO_{loc}(\sE),
\end{equation}
and it is an isomorphism.
\end{lemma}

Note that the $D_M$-module tensor product ensures that Lagrangian densities which are total derivatives give zero as functionals.

We would like to resolve $\Dens_M$ as a right $D_M$-module. 

\begin{definition}
The \textbf{$D_M$-de Rham complex} is the complex of sheaves
\begin{equation}
\dmdr:=\Omega^\bullet_{M,tw}\otimes_{\cinfty_M} D_M;
\end{equation}
this is the de Rham complex for the flat vector bundle $D_M\otimes \mathfrak o $, where $\mathfrak o$ is the orientation line bundle on $M$. The complex of sheaves $\dmdr$ has the canonical structure of a right $D_M$ module. 
\end{definition}

There is a canonical map $\dmdr[n] \to \Omega^n_{M,tw}$ induced from the right $D_M$-module action
\begin{equation}
\Omega^n_{M,tw}\otimes_{\cinfty_M} D_M \to \Omega^n_{M,tw}
\end{equation}
This map is a quasi-isomorphism, since this complex is locally the Koszul resolution of $\Omega^n_{M,tw}(U)$ with respect to the regular sequence $\{\partial_i\}_{i=1}^n$ of $D_M(U)$. 
Combining this information and Lemma 6.6.4 of Chapter 5 of \cite{cost}, we have 
\begin{lemma}
\label{lem: derhamloc}
Let $\bdyM=\emptyset$. The canonical map 
\begin{equation}
\dmdr[n]\otimes_{D_M}\sO_{red}(\sJ(\sE))\to \Olocred
\end{equation}
is a quasi-isomorphism.
\end{lemma}

\subsection{Pushforwards and pullbacks of $D$-modules}

We now assume that $\bdyM\neq \emptyset$.
In the sequel, there will be a number of situations where we will be presented with a $D_M$-module but will want a $D_\bdyM$-module or vice versa;
this section is intended to give us the fluidity to move between the two notions.
We will discuss two functors, $\dpush$ and $\dpull$.
The functor $\dpush$ will take right $D_\bdyM$-modules to right $D_M$-modules, and $\dpull$ will do the opposite for \emph{left} $D$-modules.
We note that these techniques can be used in a more general situation, namely one where $\bdyM$, $M$, and $\iota$ are replaced with general smooth manifolds $X$ and $Y$ and a general smooth map $f:X\to Y$.
Though we do not explicitly make the constructions for the general setup, 
it will be clear how to make the necessary (very small) modifications for the general situation.

We make no claims of originality here.

In both constructions, we will make use of the following sheaf on $\bdyM$:
\begin{equation}
\label{eq: pullD}
 D_{\bdyM \to M}:=\cinfty_\bdyM \otimes_{\iota^{-1}(\cinfty_M)} \iota^{-1}(D_M);
\end{equation}
this sheaf possesses a manifest right $\iota^{-1}(D_M)$ action.
Let $V\to M$ momentarily denote the bundle on $M$ whose sheaf of sections is~$D_M$.
(Concretely, $V= J(\underline \RR)^\vee$.)
One observes that $D_{\bdyM\to M}$ is simply the sheaf of sections of~$\iota^*(V)$.
From this characterization of $D_{\bdyM \to M}$, one finds also a natural left $D_{\bdyM}$ action on $D_{\bdyM \to M}$, using the pullback connection on~$\iota^*(V)$.
In the language of Equation \ref{eq: pullD}, given $X\in \textrm{Vect}(\bdyM)$, $f\in \cinfty_\bdyM$, and $D\in \iota^{-1}(D_M)$, we define
\begin{equation}
\label{eq: pullbackconnection}
X\cdot \left( f\otimes D \right):= (Xf)\otimes D +f\otimes (\widehat X D),
\end{equation}
where $\widehat X$ is a vector field on $M$ such that 
\begin{equation}
\widehat X \mid_\bdyM = X
\end{equation}
in the space $\Gamma(\bdyM,\iota^* TM)$.
One checks readily that this definition does not depend on $\widehat X$ and respects the tensor product over $\iota^{-1}(\cinfty_M)$.

\subsubsection{Pushforward}
Let $P$ be a right $D_{\bdyM}$ module. We may define 
\begin{equation}
\label{eq: pushM}
\dpush P = \iota_*\left(P\otimes_{D_\bdyM} D_{\bdyM \to M}\right);
\end{equation}
the sheaf 
\begin{equation}
\iota_*(D_{\bdyM \to M})
\end{equation}
possesses a natural right $\iota_*\iota^{-1}D_M$ module structure, and hence a natural right $D_M$ structure via the map of sheaves of algebras $D_M\to \iota_*\iota^{-1}D_M$.
It follows that $\dpush P$ is a right $D_M$ module.

Note that, for $P=D_\bdyM$, one finds~$\dpush D_\bdyM \cong \iota_*(D_{\bdyM \to M})$.
\subsubsection{Pullback}

If $\sV$ is the sheaf of sections of a finite-rank vector bundle $V\to M$, then to give a left $D_M$-module structure on $\sV$ is to give a flat connection on $V$.
We may form the pullback $\iota^* V$ on $\bdyM$; the bundle $\iota^*V$ has a natural ``pullback'' flat connection, so that its sheaf of sections is a $D_\bdyM$-module.
In other words, for $D_M$-modules arising in this way as the sheaf of sections of a vector bundle, we have discussed a construction of pullback to $\bdyM$.
Our aim in this section is to give the analogous construction for general left $D_M$-modules.

Hence, given a left $D$ module $P$, define 
\begin{equation}
\label{eq: pullM}
\dpull P := D_{\bdyM \to M} \otimes_{\iota^{-1}(D_M)} \iota^{-1}(P).
\end{equation}
If $P$ is the left $D_M$ module $D_M$, one finds that~$\dpull D_M \cong D_{\bdyM \to M}$.

\subsubsection{A useful compatibility between $\dpull$ and $\dpush$}
In this subsection, we note 

\begin{lemma}
\label{lemma: pullpush}
Let $P$ be a right $D_\bdyM$ module and $Q$ a left $D_M$-module. There are natural isomorphisms
\begin{equation}
\dpush P \otimes_{D_M} Q \cong \iota_* \left( P\otimes_{D_\bdyM} \dpull Q\right)
\end{equation}
and
\begin{equation}
\iota^{-1}\left( \dpush P\otimes_{D_M} Q\right) \cong P \otimes_{D_{\bdyM}} \dpull Q.
\end{equation}
of sheaves on $M$ and $\bdyM$, respectively.
\end{lemma}

\subsection{$\sO_{loc}(\sE)$ in the $D$-module language}
\label{sec: OlocE}
We now proceed to generalize the discussion of Section \ref{sec: localfcnlsnobdy} to the case that $\bdyM\neq \emptyset$. There are two new subtleties that arise in this situation. 
The first is that, when $\bdyM\neq\emptyset$, Lagrangian densities which are total derivatives no longer necessarily give zero as local action functionals.
The second is that we have imposed boundary conditions on the fields, so some functionals vanish because they depend on components of the boundary value of fields which are constrained to be zero.
In this section, we address the first subtlety.
In Section \ref{sec: OlocEL}, we address the second.
To the aim of understanding this first subtlety, consider the inclusion of $\cinfty_M$ modules $\Omega^{n}_{M,tw} \hookrightarrow \cD_M$ (where $\cD_M$ is the sheaf of distributions on $M$). 
Whether or not $\bdyM = \emptyset$, $\cD_M$ has a $D_M$-module structure; however, only when $\bdyM=\emptyset$ is the subspace of smooth densities closed under this $D_M$ action.
To see this, let $\omega$ be a density on $M$, let $T_\omega$ denote the corresponding distribution on $M$, i.e.
\begin{equation}
T_\omega(f) = \int_M f\omega.
\end{equation}
Let $X$ be a vector field on $M$; then, we find upon integration by parts that
\begin{equation}
\label{eq: dmdens}
\left(T_{\omega}\cdot X\right) (f) = T_\omega(X\cdot f) = \int_M (Xf)\omega = -\int_M f L_X \omega + \int_\bdyM f\mid_{\bdyM} \iota^* i_X \omega.
\end{equation}
From the last term in this equation, it follows that $(T_\omega)\cdot X$ is no longer in general a smooth density on $M$ if $\bdyM\neq \emptyset$.

On the other hand, $\Omega^n_{M,tw}$ does possess a $D_M$-module structure, even when $\bdyM\neq \emptyset$, namely the one induced from the action of vector fields on smooth densities by Lie derivatives.
(In other words, this is the action obtained by keeping the first of the two terms on the right hand side of Equation \eqref{eq: dmdens}.)
We emphasize that this is \emph{not} the $D_M$-module structure relevant for the construction of local functionals.
To see this, note that if $D_1\cdot\ldots \cdot D_k \in \sO(\sJ(\sE))$ and $\mu\in \Omega^{n}_{M,tw}$, we would like to understand $\mu\otimes D_1\cdot \ldots \cdot D_k$ as the functional
\begin{equation}
\int_M \mu\, D_1\phi \,\cdot \ldots \cdot \, D_k \phi\,;
\end{equation}
in other words, densities appear in local functionals in their capacity as distributions on $M$, whether or not $\bdyM=\emptyset$.
Of course, when $\bdyM=\emptyset$, Equation \eqref{eq: dmdens} shows that this $D_M$-module structure coincides with the one obtained by viewing densities as distributions.

We therefore seek the minimal sub-$D_M$-module of $\cD_M$ containing $\Omega^n_{M,tw}$.
That is the purpose of the following definition.

\begin{definition}
The sheaf of \textbf{$\partial$-smooth densities} $\delsmoothdistr$ is the smallest right sub-$D_M$ module of $\cD_M$ containing the smooth densities. 
\end{definition}

Let us come to a better understanding of the sheaf~$\delsmoothdistr$. 
This sheaf certainly contains the sheaf~$\Omega^n_{M,tw}$.
By Equation \eqref{eq: dmdens}, we see that it also contains the sheaf~$\iota_*\left(\Omega^{n-1}_{\bdyM,tw}\right)$, and therefore also the $D_M$-module this sheaf generates in $\cD_M$.
This last $D_M$ module is characterized in the following Lemma:

\begin{lemma}
\label{lem: bdydens}
The right $D_M$-module generated by $\iota_*\left(\Omega^{n-1}_{\bdyM,tw}\right)$ inside of $\cD_M$ is isomorphic to the right $D_M$-module
\begin{equation}
F^0 \delsmoothdistr:= \dpush \Omega^{n-1}_{\bdyM,tw}
\end{equation}
(see Equation \ref{eq: pushM} for the notation~$\dpush$).
\end{lemma}

\begin{remark}
The notation $F^0\delsmoothdistr$ is designed to suggest that this submodule is the zeroth component of a filtration on $\delsmoothdistr$. We will see later that this is indeed the case.
\end{remark}

\begin{proof}
Let us first construct a map 
\begin{equation}
\Upsilon: F^0 \delsmoothdistr\to \delsmoothdistr.
\end{equation}
Let $U\subset M$, with $V=U\cap \bdyM$. Let $\nu\in \Omega^{n-1}_{\bdyM,tw}(V)$, $f\in \cinfty_{\bdyM}(V)$, $g\in \cinfty_{c,M}(U)$, and let $D$ be the germ of a differential operator on a neighborhood of $V$ in $M$. Define
\begin{equation}
\Upsilon(\nu\otimes f\otimes D)(g) =\int_\bdyM (Dg)\mid_{\bdyM}f\nu.
\end{equation}
The map $\Upsilon$ is manifestly a sheaf map, assuming that it is well-defined.
Let us show that it is well-defined, i.e. that the relations imposed by the tensor products in $F^0\delsmoothdistr$ go to zero in $\delsmoothdistr$.
To this end, suppose that $h$ is a germ of a smooth function in a neighborhood of $V$ in $M$. Then, it is straightforward to verify that
\begin{equation}
\Upsilon( \nu\otimes f(h\mid_{\bdyM})\otimes D- \nu\otimes f\otimes hD)(g) = 0.
\end{equation}
It is likewise easy to verify that $\Upsilon$ is linear over $\cinfty_\bdyM\subset D_{\bdyM}$ and respects the right $D_M$-module structures.
Finally, let $X\in Vect(\bdyM)$ and choose some extension $\widehat X\in Vect(M)$ as in the discussion preceding Equation \eqref{eq: pullbackconnection}.
Then,
\begin{align}
\Upsilon(\nu\cdot X\otimes &f\otimes D -\nu \otimes X\cdot f\otimes D-\nu\otimes f\otimes \widehat X D)(g)
\nonumber\\&=
-\int_\bdyM (Dg)\mid_\bdyM f L_X\nu-\int_\bdyM (Dg)\mid_{\bdyM} L_X f\nu - \int_\bdyM (\widehat X Dg)\mid_{\bdyM} f\nu.
\end{align}
One checks that $(\widehat X Dg)\mid_{\bdyM} = X\cdot (Dg\mid_{\bdyM})$, so that we have
\begin{align}
\Upsilon(\nu\cdot X\otimes &f\otimes D -\nu \otimes X\cdot f\otimes D-\nu\otimes f\otimes \widehat X D)(g)
\nonumber\\&= -\int_\bdyM L_X\left( (Dg)\mid_{\bdyM} f\nu\right) = 0.
\end{align}
We have therefore established that $\Upsilon$ is a well-defined sheaf map.

It remains to show that $\Upsilon$ is a monomorphism, and that the image of $F^0\delsmoothdistr$ under $\Upsilon$ is the smallest submodule of $D_M$ containing the densities on $\bdyM$.
Choosing a tubular neighborhood (and a normal coordinate $t$), one obtains an isomorphism $D_{\bdyM \to M}\cong D_\bdyM[\partial_t]$ (of sheaves of $\cinfty_{\bdyM}$ modules on $\bdyM$).
Here, by $D_\bdyM[\partial_t]$, we mean the sheaf $D_\bdyM\otimes_\RR{\RR[\partial_t]}$.
One therefore finds that 
\begin{equation}
\label{eq: noncanonicaliso1}
F^0\delsmoothdistr \cong \iota_*\left( \Omega^{n-1}_{\bdyM, tw}[\partial_t]\right).
\end{equation}
This is an isomorphism of $\cinfty_M$-modules (in any case, we have not even described a $D_M$-module structure on the right-hand side of the equation).
The composite of this isomorphism with $\Upsilon$ takes $\nu \partial_t^n$ to the distribution (on $M$)
\begin{equation}
\label{eq: codim1dens}
g\mapsto \int_\bdyM (\partial_t^n g)\mid_{t=0}\nu. 
\end{equation}
From this characterization of $\Upsilon$, it is clear that $\Upsilon$ is a monomorphism.

Finally, we verify that $F^0\delsmoothdistr$ is the smallest sub-$(D_M)$-module of $\cD_M$ containing the densities on $\bdyM$.
Here, again, the isomorphism of Equation \ref{eq: noncanonicaliso1} is useful. Any sub $D_M$-module of $\cD_M$ which contains distributions of the form
\begin{equation}
g\mapsto \int_{\bdyM} g\mid_{t=0}\nu
\end{equation}
must also contain distributions of the form 
\begin{equation}
g\mapsto \int_{\bdyM}(\partial^n_tg)\mid_{t=0}\nu,
\end{equation}
by $D_M$-closedness.
But, Equation \ref{eq: noncanonicaliso1} tells us that $F^0\delsmoothdistr$ is precisely the space of all such distributions.

We have therefore seen that any sub-$D_M$-module of $\cD_M$ which contains $\Omega^{n-1}_{tw, \bdyM}$ also contains $F^0\delsmoothdistr$; since $F^0\delsmoothdistr$ is a sub-$D_M$-module of $\cD_M$ (via $\Upsilon$), this completes the proof.
\end{proof}

As in Notation \ref{not: lclfcnls}, let $\sO_{loc}(\sE)$ denote the space of functions on $\sE_c$ which are of the same form as those which define $\Oloc$ (the only difference being we do not identify two such if they agree on $\condfieldscs$). 

\begin{lemma}
There is a natural map 
\begin{equation}
 \digamma: \delsmoothdistr\otimes_{D_M}\sO(\sJ(\sE))\to \sO_{loc}(\sE)
\end{equation}
of $\cinfty_M$-modules.
\end{lemma}

\begin{proof}
The map is constructed just as in the case $\bdyM=\emptyset$.
More precisely, one can identify $\sO(\sJ(\sE))$ with the space of polydifferential operators on $E$ with values in smooth functions on $M$.
In other words, $\sO(\sJ(\sE))$ consists of spaces of operators which take in some number of sections of $E$ and produce a function on $M$.
We write such an object as $D_1\cdot D_2\cdot \cdots \cdot D_k$.
We define
\begin{equation}
\digamma (\omega \otimes D_1\cdot \cdots \cdot D_k) = \omega\left( D_1\phi \cdots D_k\phi\right).
\end{equation}
Here, we are viewing $\omega$ as a distribution on $M$.
(The space $\delsmoothdistr$ consists of distributions of a particular form.)
The map $\digamma$ respects the $D_M$-module tensor product for the same reason that it did in the case $\bdyM=\emptyset$.
There, it was important to use the right $D_M$-module structure on $\Omega^n_{tw, M}$ which is induced from the embedding $\Omega^n_{tw,M}\hookrightarrow \cD_M$.
By defining the $\delsmoothdistr$ as we have---namely as a subspace of the distributions on $M$---we have guaranteed that the same argument applies here.
It remains only to check that when $\omega\in \delsmoothdistr$, 
\begin{equation}
\omega\left( D_1\phi \cdots D_k\phi\right)
\end{equation}
is a local action functional.
To this end, our description of $\delsmoothdistr$ in the preceding lemma will be useful.
The distribution $\omega$ contains summands of two forms. Summands of the first type are associated to a smooth density $\mu$ on $M$ and manifestly produce local action functionals.
Summands of the second type are associated to a smooth density $\nu$ on $\bdyM$ and a natural number $n$ and are of the form in Equation \ref{eq: codim1dens}.
Such summands produce functionals of the form
\begin{equation}
I_{\nu, n}(\phi)=\int_\bdyM \nu \partial_t^n\mid_{t=0} \left( D_1\phi\cdots D_k\phi\right).
\end{equation}
Note that such a description requires a choice of tubular neighborhood $\tubnhd\cong \bdyM\times [0,\epsilon)$ of $\bdyM$ in $M$.
We continue to use this choice as follows. 
Let $f_0$ be a compactly-supported function on $[0,\epsilon)$ which is 1 at $t=0$.
Together with $\nu$, such a function determines a density on $M$ which we will call $(f_0 dt)\wedge\nu$.
It is straightforward to check that the functional
\begin{align}
J_{\nu, n}:&=\int_M \frac{\del}{\del t}\left( f_0 dt \wedge \nu \left(\frac{\del}{\del t}\right)^n\left(D_1\phi\cdots D_k \phi\right) \right)\\
&= \int_M \frac{df_0}{dt}\wedge \nu \left( \frac{\del}{\del t}\right)^n \left(D_1\phi \cdots D_k \phi\right)\nonumber\\
&+\int_M f_0 dt \wedge \nu \left(\frac{\del}{\del t}\right)^{n+1}\left( D_1\phi \cdots D_k\phi\right)
\end{align}
coincides with $I_{\nu, n}$.
Since $J_{\nu, n}$ is manifestly local in our definition, the lemma follows.
\end{proof}

\begin{lemma}
\label{lem: dmoddescription}
The natural map 
\begin{equation}
 \digamma: \delsmoothdistr\otimes_{D_M}\sO(\sJ(\sE))\to \sO_{loc}(\sE)
\end{equation}
is an isomorphism of sheaves on $M$.
\end{lemma}
\begin{proof}
Both the domain and codomain of $\digamma$ are filtered as follows. $F^0\sO_{loc}(\sE)$ is the space of all action functionals which may be written as integrals over $\bdyM$, and $F^1\sO_{loc}(\sE):=\sO_{loc}(\sE)$.
Similarly, we set 
\begin{equation}
F^0\left(\delsmoothdistr\otimes_{D_M}\sO(\sJ(\sE))\right) = (F^0\delsmoothdistr)\otimes_{D_M}\sO(\sJ(\sE)),
\end{equation}
and 
\begin{equation}
F^1\left(\delsmoothdistr\otimes_{D_M}\sO(\sJ(\sE))\right)=\delsmoothdistr\otimes_{D_M}\sO(\sJ(\sE)).
\end{equation}
The map $\digamma$ manifestly preserves these filtrations.
The associated graded object for $\sO_{loc}(\sE)$ is the direct sum 
\begin{equation}
\textrm{LMB}\oplus F^0\sO_{loc}(\sE),
\end{equation}
where $\textrm{LMB}$ is the space of local functionals modulo those which arise as integrals over the boundary.
On the other hand, the associated graded object for the domain of $\digamma$ is 
\[(\Omega^n_{M,tw}\oplus F^0\Omega^{n,n-1}_{M,tw})\otimes_{D_M} \sO(\sJ(\sE)),\] 
where $\Omega^{n}_{M,tw}$ has the right $D_M$-module structure given by Lie derivative along vector fields.
The map $\digamma$ is an isomorphism on the $F^1/F^0$ components by the same argument as for the case $\bdyM=\emptyset$.
To complete the proof, therefore, we need to argue that the $\digamma$ is an isomorphism on the $F^0$ pieces.
To this end, we use Lemma \ref{lemma: pullpush} to find:
\begin{align}
F^0 \Omega^{n,n-1}_{\bdyM,tw}\otimes_{D_M}\sO(\sJ(\sE))&= \dpush \Omega^{n-1}_{\bdyM, tw} \otimes_{D_M} \sO(\sJ(\sE))\nonumber
\\&\cong \iota_*\left( \Omega^{n-1}_{\bdyM,tw}\otimes_{D_{\bdyM}} \dpull \sO(\sJ(\sE))\right).
\end{align}
And now, we proceed as in the case $\bdyM=\emptyset$; the $D_\bdyM$-module tensor product now tells us that total derivatives on $\bdyM$ are zero.
The only difference here is that we use $\dpull \sO(\sJ(\sE))$ instead of $\sO(\sJ(\sF))$ for some vector bundle $F\to \bdyM$.
However, this precisely matches the fact that $F^0\sO_{loc}(\sE)$ contains functionals which depend on the normal derivatives of fields at the boundary.
\end{proof}

In fact, we also have a corresponding statement for the \emph{derived} $D_M$ module tensor product:

\begin{lemma}
\label{lem: derivedisunderived}
Let $\sO_{red}(\sJ(\sE)):=\sO_{red}(\sJ(\sE))/\cinfty_M$; then $\sO_{red}$ is a flat left $D_M$-module and hence 
\begin{equation}
\delsmoothdistr\otimes_{D_M}\sO_{red}(\sJ(\sE))\simeq  \Omega_{M,tw}^{n,n-1}\otimes_{D_M}^\LL\sO_{red}(\sJ(\sE))
\end{equation}
and there is a quasi-isomorphism
\begin{equation}
\delsmoothdistr\otimes_{D_M}^\LL\sO(\sJ(\sE))\simeq \sO_{loc,red}(\sE)
\end{equation}
of sheaves (of $\cinfty_M$ modules).
\end{lemma}

\begin{proof}
The proof is exactly the same as the argument used in the proof of Lemma 6.6.2 of \cite{cost}.
\end{proof}

Since Lemma \ref{lem: derivedisunderived} holds, we can replace $\delsmoothdistr$ with a quasi-isomorphic $D_M$ module, and we would have a quasi-isomorphic model of $\sO_{loc,red}(\sE)$. Hence, we seek a resolution of $\delsmoothdistr$ as a right $D_M$-module.
This is what we proceed to do now.
As before, the left $D_M$-module structure on $D_M$ provides a differential on the graded right $D_M$-module
\begin{equation}
\dmdr:=\Omega^\bullet_{M,tw}\otimes_{\cinfty_M} D_M;
\end{equation}
a similar argument endows the graded right $D_M$-module
\begin{equation}
\dmdrbdy:=\iota_*\left(\Omega^\bullet_\bdyM\otimes_{\cinfty_\bdyM}D_{\bdyM\to M}\right)
\end{equation}
with a differential, and a natural right $D_M$-module structure via the map $D_M\to \iota_*\iota^{-1} D_M$.
There is a map
\begin{equation}
\label{eq: pullbackdeRham}
\iota^*: \dmdr \to \dmdrbdy
\end{equation}
defined using the pullback of forms and the unit map $D_M \to \iota_*\iota^{-1}D_M$.

\begin{definition}
The \textbf{universal de Rham resolution of }$\delsmoothdistr$ is the shifted mapping cone
\begin{equation}
\dmdrrel := \cone(\iota^{*})[n-1]
\end{equation}
of the map $\iota^*$.
\end{definition}

The following lemma justifies our introduction of $\dmdrrel$.

\begin{lemma}
\label{lem: drres}
There is a natural quasi-isomorphism
\begin{equation}
\Phi:\dmdrrel \to \delsmoothdistr
\end{equation}
of right $D_M$-modules.
\end{lemma}

\begin{proof}
Given $\alpha \in \Omega^n_{tw,M}$, $\beta\in \Omega^{n-1}_{tw,\bdyM}$, $D_1,D_2\in D_M$, we let
\begin{equation}
\Phi(\alpha\otimes D_1, \beta\otimes D_2) = \alpha\cdot D_1 + \beta \cdot D_2;
\end{equation}
on the right-hand side, we understand $\alpha$ and $\beta$ as elements of $\delsmoothdistr$ and $\cdot$ denotes the right $D_M$-module action in this space. This is manifestly a map of right $D_M$-modules, and $\Phi$ respects the tensor product relations defining $\dmdr$ and $\dmdrbdy$. 
Let us check that $\Phi$ respects the differentials. Namely, we are required to check that if $\mu\in \Omega^{n-1}_{M,tw}, \nu\in \Omega^{n-2}_{M,tw}$, then 
\begin{equation}
\Phi(d\mu \otimes D_1+\mu \otimes \nabla D_1,-\iota^*\mu \otimes D_1+d\nu\otimes D_2+ \nu \otimes \nabla D_2) = 0.
\end{equation}
Because the differential and $\Phi$ are right $D_M$-linear (the connection $\nabla$ on $D_M$ is defined in terms of \emph{left} multiplication in $D_M$), it suffices to check this for $D_1=D_2=1$. Then, one verifies that the required equality is equivalent to integration by parts formulas for $M$ and $\bdyM$. 

It remains to check that $\Phi$ is a quasi-isomorphism. Note that, as before, we obtain a natural filtration on $\dmdrrel$ and $\Phi$ preserves this filtration. On the level of associated graded sheaves, $Gr(\Phi)$ is the sum of the map
\begin{equation} 
\dmdr [n]\xrightarrow{\varepsilon} \Omega^{n}_{M,tw}
\end{equation} 
(where $\Omega^{n}_{M,tw}$ is endowed with the $D_M$-module structure by Lie derivatives of vector fields) and the map
\begin{equation}
\dmdrbdy\xrightarrow{\varepsilon_\partial} F^0 \delsmoothdistr.
\end{equation}
The map $\varepsilon$ is precisely the quasi-isomorphism constructed in the paragraph preceding Lemma \ref{lem: derhamloc}.
For the story on $\bdyM$, let us recall that 
\begin{equation}
F^0 \delsmoothdistr:=\iota_*\left(\Omega^{n-1}_{\bdyM,tw}\otimes_{D_{\bdyM}}\left( \cinfty_{\bdyM}\otimes_{\iota^{-1}(\cinfty_M)}\iota^{-1}(D_M)\right)\right)
\end{equation}
and 
\begin{equation}
\dmdrbdy:= \iota_*\left( \Omega^\bullet_{\bdyM,tw}\otimes_{\iota^{-1}(\cinfty_M)} \iota^{-1}(D_M)\right).
\end{equation}
Since the pushforward functor $\iota_*$ is exact, it suffices to check that the map 
\begin{equation}
\Phi_\del:=\Omega^\bullet_{\bdyM,tw}\otimes_{\iota^{-1}(\cinfty_M)} \iota^{-1}(D_M)\to 
\Omega^{n-1}_{\bdyM,tw}\otimes_{D_{\bdyM}}\left( \cinfty_{\bdyM}\otimes_{\iota^{-1}(\cinfty_M)}\iota^{-1}(D_M)\right)
\end{equation}
induced from $\Phi$ is a quasi-isomorphism.
Upon choosing a tubular neighborhood of the boundary, one may filter both the domain and codomain of $\Phi_\del$ by the order of the differential operator in the normal direction (it is possible to show that this filtration does not depend on this choice).
Moreover, one has isomorphisms
\begin{equation}
Gr\left(\Omega^\bullet_{\bdyM,tw}\otimes_{\iota^{-1}(\cinfty_M)} \iota^{-1}(D_M)\right) \cong \Omega^\bullet_{\bdyM,tw}\otimes_{\cinfty_\bdyM}D_{\bdyM}[\del_t],
\end{equation}

\begin{equation}
Gr\left( \Omega^{n-1}_{\bdyM,tw}\otimes_{D_{\bdyM}}\left( \cinfty_{\bdyM}\otimes_{\iota^{-1}(\cinfty_M)}\iota^{-1}(D_M)\right)\right) \cong \Omega^{n-1}_{\bdyM,tw}[\partial_t],
\end{equation}
with the na\"ive right $\iota^{-1}(D_M)$-module structures (i.e. one decomposes a vector field $X$ as $X=X_\del+X_\nu \partial_t$ near the boundary, and lets the normal part of the vector field act on the $\RR[\partial]$ factors, while the tangential part acts in the canonical way on $D_\bdyM$ and $\Omega^{n-1}_{\bdyM,tw}$, respectively).
Using this characterization, it is clear that $Gr(\Phi_\del)$ is simply the base change over $\RR[\partial_t]$ of the codimension-one version of $\epsilon$. Hence, $Gr(\Phi_\del)$, and therefore $Gr(\Phi)$, and therefore $\Phi$, is a quasi-isomorphism.
\end{proof}

The following Lemma is an immediate consequence of Lemma \ref{lem: drres}.

\begin{lemma}
\label{lem: locfcnlresolution}
There is a cannonical quasi-isomorphism
\begin{equation}
\dmdrrel\otimes_{D_M}\sO_{red}(\sJ(\sE))\to \sO_{loc,red}(\sE)
\end{equation}
of sheaves on $M$.
\end{lemma}

We describe one property of $\dmdrrel$ which will be useful in the sequel. 
\begin{lemma}
\label{lmm: dmdrrelflat}
The right $D_M$-module $\dmdrrel$ is a complex of flat $D_M$-modules.
\end{lemma}

\begin{proof}
As a right $D_M$-module, $\dmdr$ is locally a free $D_M$-module, hence is flat.

As a right $D_M$-module, $\dmdrbdy$ looks locally like a sum of modules of the form $\iota_*\iota^{-1}(D_M)$. 
The endofunctor (acting on the category of left $D_M$-modules) 
\begin{equation}
\iota_*\iota^{-1}(D_M)\otimes_{D_M}\cdot
\end{equation}
is naturally isomorphic to the endofunctor $\iota_*\iota^{-1}$.
This functor is exact, since $\iota_*$ and $\iota^{-1}$ are. (Here, we use the fact that $\bdyM$ is a closed submanifold of $M$.)

The lemma follows.
\end{proof}

\subsection{Imposing the Boundary Condition for Functionals}
\label{sec: OlocEL}
Lemma \ref{lem: locfcnlresolution} helps us to understand the structure of $\sO_{loc}(\sE)$; 
however, we are more interested in $\Oloc$, 
i.e. we are not interested in distinguishing two functionals if they agree when restricted to $\condfieldscs$. 
To this end, note that there is a surjective map $\sO_{loc}(\sE)\to \Oloc$. 
Our aim now is to characterize the kernel of this map.
Proposition \ref{prop: dualsofcondfields} and its Corollary A tell us which functionals on $\sE$ vanish when restricted to $\condfields$, 
namely those which depend in one of their inputs only on the boundary information of the field's value in $\Eb/L$. 
Let us use this information to get a better grasp on $\Oloc$. 

Note the following: if a local functional depends only on the boundary information of one input, then it lies in the space $F^0 \Oloc$, i.e. is described by an integral over $\bdyM$.
This is because local functionals of order $k$ in the fields have their support on the small diagonal of $M^k$.
Within $\dmdrrel\otimes \sO_{red,loc}(\sJ(\sE))$ (which, by Lemma \ref{lem: locfcnlresolution}, resolves $\sO_{loc,red}(\sE)$), the space of functionals which are supported on $\bdyM$ is 
\begin{equation}
\iota_*( \Omega^\bullet_{\bdyM, tw}(D_{\bdyM\to M}))\otimes_{D_{M}}\sO_{red}(\sJ(\sE)).
\end{equation}
By Lemma \ref{lemma: pullpush}, we have
\begin{equation}
\iota_*( \Omega^\bullet_{\bdyM, tw}(D_{\bdyM\to M}))\otimes_{D_{M}}\sO_{red}(\sJ(\sE))\cong \iota_*\left(\Omega^\bullet_{\bdyM, tw}\otimes_{\cinfty_{\bdyM}}\sO_{red}(\dpull (\sJ(\sE)))\right),
\end{equation}
where $\sO_{red}(\dpull(\sJ(\sE)))$ is defined similarly to $\sO_{red}(\sJ(\sE))$, but by taking continuous duals and symmetric powers of $\dpull \sJ(\sE)$ over $\cinfty_\bdyM$.

As discussed in \autocite{cost}, $\sJ(\sE[-1])$ is an $L_\infty$-algebra in the category of $D_M$-modules.
The $L_\infty$ brackets on $\sE[-1]$ induce the brackets on $\sJ(\sE[-1])$, and the requirement that the brackets on $\sE[-1]$ be polydifferential operators guarantees that the induced brackets on $\sJ(\sE[-1])$ are maps of $D_M$-modules.
Throughout, we have implicitly endowed $\sO(\sJ(\sE))$ with the Chevalley-Eilenberg differential induced from this $L_\infty$-algebra structure.
The maps we have constructed so far, relating different descriptions of the local functionals, have intertwined the Chevalley-Eilenberg differentials on the respective models.
We mention this Chevalley-Eilenberg differential now because we proceed with a construction for which it is necessary to verify that the Chevalley-Eilenberg differential is preserved.
There is a canonical map of shifted $L_\infty$-algebras in the category of $D_{\bdyM}$-modules
\begin{equation}
\Upsilon: \dpull (\sJ(\sE)) \to \sJ(\sL')
\end{equation}
constructed as the composite of the following two maps.
First, we construct a map
\begin{equation}
\Upsilon_1: \dpull (\sJ(\sE))=\cinfty_\bdyM \otimes_{\iota^{-1}(\cinfty_M)}\iota^{-1}(\sJ(\sE)) \to \sJ(\sEb)
\end{equation}
given by
\begin{equation}
\Upsilon_1(f\otimes \sigma) = f\rho(\sigma);
\end{equation}
it is straightforward to check that $\Upsilon_1$ is well-defined (using the fact that $\rho$ arises from restriction to the boundary followed by a bundle operation).
and moreover respects the tensor product defining $\dpull (\sJ(\sE))$.
Another way to describe $\Upsilon_1$ is as follows.
Note that $\dpull(\sJ(\sE))$ is the sheaf of sections of $J(E)|_\bdyM$.
There is a natural bundle map
\[
J(E)|_\bdyM \to J(E|_\bdyM)
\]
(which ``forgets normal jets'') and a natural bundle map
\[
J(E|_\bdyM)\to J(\Eb).
\]
The map $\Upsilon_1$ is induced from the composite of these two maps on the sheaves of sections.
Next, we consider the map
\begin{equation}
\Upsilon_2: \sJ(\sEb)\to \sJ(\sL')
\end{equation}
induced from the bundle map $\Eb\to L'$, and we set $\Upsilon =\Upsilon_2\circ \Upsilon_1$.
It follows that there is an inclusion
\begin{equation}
\Upsilon^\vee:\sJ(\sL')^\vee \to (\dpull \sJ(\sE))^\vee
\end{equation}
of left $D_\bdyM$ modules. Because $\Upsilon$ is a map of shifted $L_\infty$-algebras, it follows that 
\begin{equation}
\sJ(\sL')^\vee\otimes_{\cinfty_\bdyM} \sO(\dpull(\sJ(\sE)))\subset \sO_{red}(\dpull (\sJ(\sE)))
\end{equation}
is closed under the Chevalley-Eilenberg differential on $\sO_{red}(\dpull(\sJ(\sE)))$. 
The inclusion also respects the $D_\bdyM$ module structure, and it is straightforward to show that this sub $D_\bdyM$-module is flat. Hence, we can make the following definition:

\begin{definition}
The complex of \textbf{de Rham functionals vanishing on $L$} is the complex
\begin{equation}
\sV(L): = \iota_*\left( \Omega^\bullet_\bdyM \otimes_{\cinfty_{\bdyM}} \sJ(\sL')^\vee\otimes_{\cinfty_\bdyM} \sO_{red}(\dpull (\sJ(\sE)))\right).
\end{equation}
The discussion preceding this definition guarantees that $\sV(L)$ is a subcomplex of $\dmdrrel\otimes_{D_M}\sO_{red}(\sJ(\sE))$. 
\end{definition}

\begin{notation}
We define the following complex of sheaves
\begin{equation}
\fullOloc:=\left(\dmdrrel\otimes_{D_M}\sO_{red}(\sJ(\sE))\right)/\sV(L)
\end{equation}
\end{notation}

\begin{lemma}
There is a quasi-isomorphism of complexes of sheaves
\begin{equation}
\fullOloc\to \Olocred.
\end{equation}
\end{lemma}

\begin{proof}
Since $\dmdrpurebdy\to \Omega^{n-1}_{\bdyM,tw}$ is a quasi-isomorphism of right $D_\bdyM$-modules and $\sJ(\sL')^\vee\otimes_{\cinfty_\bdyM} \sO(\dpull(\sJ(\sE)))$ is a flat left $D_\bdyM$-module, we have a quasi-isomorphism
\begin{equation}
\label{eq: vanishingfcnls}
\sV(L)\to \iota_*\left(\Omega^{n-1}_{\bdyM,tw}\otimes_{D_\bdyM} \sJ(\sL')^\vee\otimes_{\cinfty_\bdyM} \sO(\dpull(\sJ(\sE)))\right).
\end{equation}
Using Corollary A of the appendix, we identify this latter space of functionals as precisely the space of local functionals on $\sE_c$ which vanish when restricted to $\condfieldscs$.
Hence, we have a map of short exact sequences:

\begin{equation}
\xymatrix{
0 \ar[r] & \sV(L) \ar[r]\ar[d]& A_1\ar[r]\ar[d] & \fullOloc\ar[r]\ar[d]&0\\
0 \ar[r] &A_2  \ar[r]& \sO_{loc,red}(\sE)\ar[r]& \Olocred\ar[r]&0
},
\end{equation}
where 
\begin{equation}
A_1 = \dmdrrel\otimes_{D_M}\sO_{red}(\sJ(\sE))
\end{equation}
and 
\begin{equation}
A_2 = \iota_*\left(\Omega^{n-1}_{\bdyM,tw}\otimes_{D_\bdyM} \sJ(\sL')^\vee\otimes_{\cinfty_\bdyM}\sO(\dpull(\sJ(\sE)))\right).
\end{equation}
Since the first two vertical maps are quasi-isomorphisms, it follows by the snake lemma and the five lemma that the last vertical arrow is a quasi-isomorphism as well.
\end{proof}

\begin{remark}
We have implicitly also found a description of the space $\sO_{loc}(\condfields)$ before taking the de Rham resolution of densities: it is simply the quotient of $\sO_{loc}(\sE)$, as described in Lemma \ref{lem: dmoddescription}, by the codomain of Equation \ref{eq: vanishingfcnls}.
\end{remark}

\subsection{BF Theory, an Extended Example}
\label{subsec: bfthyloc}
The preceding discussion is a bit dense and difficult to follow. 
To help the reader to get a sense of these constructions, we now make an extended study of BF theory on the upper half-space $\HH^n$.
We will study BF theory with the $A$ and $B$ boundary conditions (the terminology follows Example \ref{ex: bfbdycond}).

As in the preceding discussion, we will start by understanding $\sO_{loc,red}(\sE)$. Recall (Lemma \ref{lem: locfcnlresolution}), that there is a quasi-isomorphism

\begin{equation}
\dmdrrel\otimes_{D_M}\sO_{red}(\sJ(\sE))\to \sO_{loc,red}(\sE).
\end{equation}

One can use an argument from the Proposition 7.6 of \autocite{rabaxial} (although this argument is itself derivative from other sources) to construct a quasi-isomorphism of left $D_M$-modules
\begin{equation}
\sO_{red}(\sJ(\sE)) \overset{\sim}{\to} \cinfty_M \otimes C^\bullet_{red}(\fg, \Sym(\fg[n-2])).
\end{equation}
Hence, using Lemma \ref{lmm: dmdrrelflat}, one has a  zigzag of quasi-isomorphisms
\begin{equation}
\dmdrrel\otimes_{D_M}\cinfty_M\otimes \left(C^\bullet_{red}(\fg, \Sym(\fg[n-2]))\right) \to \sO_{loc,red}(\sE).
\end{equation}
It is straightforward to show that
\begin{equation}
\dmdrrel\otimes_{D_M}\cinfty_M \cong \left(\Omega^\bullet_{M}[n]\overset{\iota^*}{\to}\iota_*\Omega^\bullet_{\bdyM}[n-1]\right)=:\drrel;
\end{equation}
hence, we have constructed a zigzag of quasi-isomorphisms
\begin{equation}
\drrel \otimes C^\bullet_{red}(\fg, \Sym (\fg[2-n]))\to \sO_{loc,red}(\sE).
\end{equation}

Now, let us understand $\sV(L)$ better. Let us take first the $A$ boundary condition.
Here, $L'=\Lambda^\bullet T^* \bdyM\otimes \fg^\vee[n-2]$.
As before, we can show 
\begin{equation}
\sJ(\sL')\simeq \cinfty_\bdyM\otimes \fg^\vee[n-2]
\end{equation}
and 
\begin{equation}
\dpull(\sJ(\sE))\simeq \cinfty_\bdyM \otimes (\fg[1]\oplus \fg^\vee[n-2]).
\end{equation}
Hence, we have a zigzag of quasi-isomorphisms
\begin{equation}
\iota_*\left(\Omega^\bullet_{\bdyM}\right)\otimes \left( \fg[2-n]\otimes C^\bullet(\fg, \Sym(\fg[2-n])\right)\to \sV(L).
\end{equation}
Let 
\begin{equation}
\sO(\fg;A)
\end{equation}
denote the total complex
\begin{equation}
\Omega^\bullet_{M}[n]\otimes C^\bullet_{red}(\fg, \Sym (\fg[n-2]))\to \iota_*(\Omega^\bullet_{\bdyM})[n-1]\otimes C^\bullet_{red}(\fg),
\end{equation}
where the horizontal map is the tensor product of the maps
\begin{equation}
\iota^* :\Omega^\bullet_M \to \iota_*(\Omega^\bullet_{\bdyM})
\end{equation}
and the natural projection
\begin{equation}
C^\bullet_{red}(\fg, \Sym(\fg[n-2]))\to C^\bullet_{red}(\fg).
\end{equation}
We have thus constructed a zigzag of quasi-isomorphisms
\begin{equation}
\sO(\fg; A) \to \Olocred
\end{equation}
for BF theory with the $A$ boundary condition.
Similarly, define
\begin{equation}
\sO(\fg; B):=\Omega^\bullet_{M}[n]\otimes C^\bullet_{red}(\fg, \Sym (\fg[n-2]))\to \iota_*(\Omega^\bullet_{\bdyM})[n-1]\otimes \Sym^{\geq 1}(\fg[n-2]);
\end{equation}
we can construct a zigzag of quasi-isomorphisms
\begin{equation}
\sO(\fg; B) \to \Olocred
\end{equation}
for BF theory with the $B$ boundary condition.

Let us now specialize to the case $M=\HH^2$, $\bdyM = \RR$ and take global sections. In this case, we have equivalences
\begin{align}
H^\bullet(\sO(\fg;B)(\HH^2)) &\simeq \left( C^\bullet_{red}(\fg, \Sym(\fg))[2] \to \Sym^{\geq 1}(\fg)[1]\right)\\
&\simeq \left(H^{\geq 1}(\fg, \Sym(\fg))\right)[2] \oplus \left(\Sym^{\geq 1}(\fg)/\Sym^{\geq 1}(\fg)^{\fg}\right)[1].
\end{align}
When $\bdyM=\emptyset$, the zeroth and first cohomology of the complex of local functionals represent the space of deformations and obstructions, respectively, for quantizations of classical BF theory.
Adopting the same terminology here without proving the analogous statement, we find that the space of obstructions for quantizations of 2D BF theory with $B$ boundary condition is $H^3(\fg, \Sym(\fg))$, and the space of deformations for 2D BF theory is $H^2(\fg, \Sym(\fg))$.
The cohomology acquires a grading from the symmetric degree in $\Sym(\fg)$; we will call this the \emph{B-weight}.
If we restrict to those functionals which have a B-weight of $+1$, and further assume that $\fg$ is semi-simple, we find that 
\begin{equation}
H^\bullet\left( \sO(\fg; B)(\HH^2)\right)_{1} \cong \fg[1],
\end{equation}
the subscript $1$ reminding us that we are computing the cohomology in $B$-weight 1.
To perform this computation, we use Whitehead's theorem and the fact that semi-simple Lie algebras have trivial centers.
Given $x\in \fg$, one may exhibit the isomorphism just described via the $B$-field functional
\begin{equation}
J_x(\beta) = \left(\int_\RR \iota^* \beta\right)(x).
\end{equation}
It is straightforward to check that this functional is closed in $\Olocred$.

%% file: AppendixA.tex
\section{A Brief Primer on Differentiable Vector Spaces}
\label{sec: DVS}
In this Appendix, we provide a brief review of the notions concerning differentiable and convenient vector spaces that we use in the body of the article.
Our main goal is to give the reader enough familiarity with the basic properties of the categories $\DVS$ and $\CVS$ of differentiable and convenient vector spaces, respectively, to understand the proofs in the main text.
We hope also that this Appendix provides some explanation as to why certain points of functional analysis are belabored in the main text.
To this end, this appendix follows---as opposed to a ``definition-lemma-theorem'' format---an ``informal statement of properties-commentary-informal statement of properties'' format.

As we have remarked in the introduction, we would like our factorization algebras to take values in a symmetric monoidal category of chain complexes of infinite-dimensional vector spaces.
We have already remarked that we would like to find an abelian category of infinite-dimensional vector spaces.
Let us discuss the desired symmetric-monoidality of our category a bit further.
Typically, the vector spaces we consider are infinite-dimensional, and endowed with topologies; the natural tensor products in this context have universal properties in categories of topological vector spaces where the morphisms are continuous linear maps.
For example, if $E\to M$ is a vector bundle, then the space of sections $\cinfty(M;E)$ possesses the so-called Whitney topology.
There is a natural tensor product on a subcategory of the category of all topological vector spaces, called the \emph{completed projective tensor product} and denoted $\hotimes_\pi$, with the property that for $F\to N$ another vector bundle, the equation 
\begin{equation}
\label{eq: cptp}
\cinfty(M; E)\hotimes_\pi \cinfty(N; F)\cong \cinfty(M\times N; E\boxtimes F)
\end{equation}
holds, where $E\boxtimes F$ is the external tensor product of vector bundles.
Such an equation simply does not hold if one takes the algebraic tensor product of the two spaces of sections.

Based on this discussion and the definition of a factorization algebra, we seek a category $\cC$ of vector spaces such that
\begin{enumerate}
\item $\cC$ is an abelian category,
\item $\cC$ is endowed with a tensor product in which equations like Equation \eqref{eq: cptp} hold.
\end{enumerate}
One standard category meant to encode the theory of infinite-dimensional vector spaces---namely, that of locally convex topological vector spaces---is infamously non-abelian. 
Moreover, we would like Equation \eqref{eq: cptp} to hold even if $M$ and $N$ are non-compact, and one takes compactly-supported sections throughout.
This, too, fails to obtain in the category of locally convex topological vector spaces with the completed projective tensor product.

The approach we follow here is to consider the category of differentiable vector spaces $\DVS$, as discussed in Appendix B of \autocite{CG1}.
The category $\DVS$ has the following properties:
\begin{enumerate}
\item It is abelian.
\item It has a multi-category structure.
\item It possesses an enrichment over itself, i.e. one can form a differentiable vector space $\underline{\DVS}(V,W)$ of linear maps between two differentiable vector spaces.
\end{enumerate}

Very briefly, a differentiable vector space $V$ is a sheaf of vector spaces $\cinfty(\cdot, V)$ (we will use both notations interchangeably to represent $V$) on the site of smooth manifolds.
Moreover $\cinfty(\cdot, V)$ must be endowed with the structure of a module over the sheaf of rings $\cinfty$ and a flat connection
\begin{equation}
\nabla: \cinfty(\cdot, V) \to \cinfty(\cdot, V)\otimes_{\cinfty} \Omega^1(\cdot).
\end{equation}
For any DVS $V$, the vector space $\cinfty(\{*\}, V)$ encodes the underlying vector space of $V$, and $\cinfty(M,V)$ encodes the space of smooth maps from $M$ to $V$.
For example, given a vector bundle $E\to M$, the differentiable vector space $\cinfty(M; E)$ is given by the assignment
\begin{equation}
\cinfty(X,\cinfty(M;E)):= \cinfty(X\times M; \pi_2^* E),
\end{equation}
where $\pi_2: X\times M\to M$ is the natural projection.
(Note that we use the notation $\cinfty(M;E)$ to denote global sections of $E\to M$, and the notation $\cinfty(X,V)$ to denote smooth maps from $X$ to $V$.)
We refer the reader to the aforementioned Appendix of \autocite{CG1} for further details.

The category $\DVS$ has all of the categorical properties that we desired in a category of infinite-dimensional vector spaces.
Nevertheless, it is somewhat difficult to work with differentiable vector spaces directly.
Hence, Costello and Gwilliam study another category of infinite-dimensional vector spaces, namely the category of \emph{convenient} vector spaces $\CVS$.
The category $\CVS$ possesses the following properties:
\begin{enumerate}
    \item The objects in $\CVS$ are vector spaces endowed with additional structure (namely a \emph{bornology}), and morphisms are spaces of bounded linear maps.
    \item The category $\CVS$ possesses all limits and colimits.
    \item The category $\CVS$ is closed symmetric monoidal. The symmetric monoidal structure on $\CVS$ is denoted by the symbol $\hotimes_\beta$, and the inner hom object is denoted $\underline{\CVS}(V,W)$.
    \item Vector spaces of smooth, compactly-supported, distributional, and compactly-supported distributional sections of a vector bundle $E\to M$ can be described by objects in $\CVS$.
    \item There is a functor $dif: \CVS\to \DVS$ which preserves all limits. Moreover, the functor preserves inner hom objects, and the symmetric monoidal product $\hotimes_\beta$ represents the multi-category structure on~$\DVS$.
    \item Equation \eqref{eq: cptp} holds, with $\hotimes_\beta$ replacing $\hotimes_\pi$, both for spaces of sections of a bundle $E\to M$ and spaces of compactly-supported sections of $E$.
    \item The functor $dif$ embeds $\CVS$ as a full subcategory of $\DVS$.
\end{enumerate}
Again, we refer the reader to Appendix B of \autocite{CG1} for further details.
It is far easier to work directly with convenient vector spaces than with differentiable vector spaces.
However, the category $\CVS$ is not abelian, so one may not use classical homological algebra to study chain complexes of convenient vector spaces.
Further, the functor $\CVS\to \DVS$ does not preserve colimits, so in the computation of homology groups, it matters whether a complex 
\[
\cdots \to V_i \to V_{i+1}\to \cdots
\]
of convenient vector spaces is considered to be a complex of convenient or differentiable vector spaces via the functor~$dif$.
We adopt the following conventions, which we follow implicitly throughout the article.
First, nearly all of our chain complexes of differentiable vector spaces will arise from chain complexes of convenient vector spaces.
To verify that a map of complexes $dif(f):dif(A^\bullet) \to dif(B^\bullet)$ is a quasi-isomorphism, we never compute the cohomology groups of the underlying complexes of convenient vector spaces $A^\bullet, B^\bullet$ directly.
Instead, we allow ourselves only to use some combination of the following tools:
\begin{enumerate}
    \item Construction of a homotopy inverse for $f$, in which case both $f$ and $dif(f)$ are quasi-isomorphisms by functoriality.
    \item Use of standard homological-algebraic techniques in $\DVS$, such as spectral sequences and the snake lemma.
\end{enumerate}
This allows us to avoid computing cohomology groups directly in $\DVS$, and instead to perform explicit computations in $\CVS$, doing so in a way that is ``kosher.''

In the next section, we will exhibit a special instance of this strategy: namely, in order to compute a quotient $V/W$, we will provide a splitting $V\cong W\oplus U$. Then, $U\cong V/W$ holds both in $\CVS$ and $\DVS$.

\section{The duals of function spaces with boundary conditions}
\label{sec: innerhom}
In the body of the text, we have defined observables in terms of an inner hom of differentiable vector spaces.
The goal of this appendix is to give an explicit characterization of these inner homs.

Let $V\to M$ be a vector bundle on a manifold with boundary $M$. 
Let $L\subset V|_\bdyM$ be a sub-bundle of the restriction of $V$ to the boundary.
We will let $\sV$ and $\sL$ denote---in a slight departure of convention from the main body of the text---the spaces of global sections of $V$ and $L$ over their respective bases.
We let $\sV_L$ denote the space of sections of $V$ over $M$ which lie in $L$ when restricted to the boundary.
We establish a similar notation, $\sV_{L,c}$, for the space of such sections which have compact support on $M$.
With the end of understanding the inner hom spaces
\begin{equation}
\label{eq: innerhom1}
\innerhom{\sV_L}{\RR}
\end{equation}
and 
\begin{equation}
\label{eq: innerhom2}
 \innerhom{\sV_{L,c}}{\RR}
\end{equation}
and their generalizations, 
we will provide splittings
\begin{align}
    \sV&\cong \sV_L \oplus \sC\\
    \sV_c &\cong \sV_{L,c}\oplus \sC_{c},
\end{align}
where $\sC$ is the space of global sections of a bundle on $\bdyM$ and $\sC_c$ its corresponding space of compactly-supported sections.
We will derive these isomorphisms in the category $\CVS$, but the functor from $\CVS$ to $\DVS$ preserves limits, so we obtain the isomorphism also in the category $\DVS$.
Note that these splittings will allow us to characterize the inner hom spaces in Equations \eqref{eq: innerhom1} and \eqref{eq: innerhom2} as quotients of the inner hom spaces
\begin{equation}
    \innerhom{\sV}{\RR}
\end{equation}
and 
\begin{equation}
    \innerhom{\sV_c}{\RR}
\end{equation}
respectively, in both $\CVS$ and $\DVS$.
(Recall that the functor $\CVS\to \DVS$ does not preserve colimits in general and hence a computation of a colimit in $\CVS$ does not suffice in general to give the colimit in $\DVS$.
In the present case, however, the comptuation in $\CVS$ does suffice since we have provided a direct sum decomposition of $\sV$ and $\sV_c$.)

Let $C$ denote the bundle $(V\mid_{\bdyM})/L$ on $\bdyM$ and $\sC$ its space of global sections. 
There are natural surjective maps 
\begin{align}
P&:\sV\to \sC\\
P_c&:\sV_c\to \sC_c;
\end{align}
In the next proposition, we construct a map $I: \sC\to \sV$ which establishes $\sC$ as a complement to $\sV_L$ in $\sV$.
\begin{proposition}
\label{prop: dualsofcondfields}
There are (non-canonical) isomorphisms
\begin{align}
    \sV&\cong \sV_L\oplus \sC\\
    \sV_c&\cong \sV_{L,c}\oplus \sC_c
\end{align}
of convenient vector spaces.
\end{proposition}

\begin{proof}
Let us focus on the first isomorphism first.
We will construct a splitting $I$ of $P$ with the property that $\im(1-IP)\subset \sV_L$. 
To see that this suffices to prove the Proposition, first note that $\sV_L= \ker P$, and consider the map 
\begin{align}
T&:\sV \to \sV_{L}\oplus \sC\\
T(v)&= ((1-IP)v,Pv).
\end{align}
$T$ is a continuous linear isomorphism, by standard arguments. $T$ has an inverse $S$ given by $i\oplus I$, where $i$ is the inclusion of $\sV_L$ as a closed subspace of $\sV$. In other words, we will have 
\begin{equation}
\label{eq: isoquot}
\sV\cong \sV_L\oplus \sC.
\end{equation}

Let us now construct the splitting $\sC\to \sV$. To this end, choose a tubular neighborhood $\tubnhd\cong \bdyM\times[0,\epsilon)$ of $\bdyM$ in~$M$. 
Let $\pi$ be the projection
\[
\tubnhd\to [0,\epsilon).
\]
By the homotopy invariance of bundles, we may assume that
\[
V|_T\, \cong \pi^* V\mid_{\bdyM}.
\]
In other words, we may assume that $V$ is trivial in the normal direction.
Let $\chi$ be a compactly-supported function on $[0,\epsilon)$ which is 1 in a neighborhood of $0$ and with support contained in $[0,\epsilon/2]$. Let $c\in \sC$ be a section of $C$. 
Choose a splitting of bundles 
\begin{equation}
    \Psi: C\to (V|_{\bdyM}).
\end{equation}
Then, we set $I(c) = \chi \Psi (c)$.
(This is where we have used the isomorphism $V|_T \, \cong \pi^*V|_{\bdyM}$.)
It is straightforward to verify that $I$ is continuous and satisfies the equations $PI=id$ and~$\im(1-IP)\subset \sV_L$.

We now construct the second isomorphism of the lemma. 
Cover $M$ by a countable collection $
\cK_1\subset \cK_2\subset\cdots$ of compact subsets. 
We may assume, by replacing $\cK_i$ with $\cK_i\cup (\cK_i\cap \bdyM)\times [0,\epsilon/2]$ for all $i$, that $\cK_i$ contains $(\cK_i\cap \bdyM)\times [0,\epsilon/2]$. Then, the formulas for $I$ and $P$ send $\sC_{\cK_i\cap \bdyM}\to \sV_{\cK_i}$ and $\sV_{\cK_i}\to \sC_{\cK_i\cap \bdyM}$. We therefore obtain an isomorphism
\begin{equation}
\sV_{\cK_i}\cong \sV_{\cK_i,L}\oplus \sC_{\cK_i\cap \bdyM}
\end{equation}
for each $i$.
The isomorphism respects the maps induced from the inclusions $\cK_i\subset \cK_{i+1}$ and $\cK_i\cap \bdyM\subset \cK_{i+1}\cap \bdyM$, and so we have also an isomorphism 
\begin{equation}
\sV_c=\mathrm{colim}_i \sV_{\cK_i}\cong \mathrm{colim}_i(\sV_{L,\cK_i})\oplus\mathrm{colim}_i \sC_{\cK_i\cap \bdyM}= \sV_{\sL,c}\oplus \sC_c.
\end{equation}
This completes the proof.
\end{proof}

The most useful application of Proposition \ref{prop: dualsofcondfields} is the following corollary.
To make the notation more compact, we make the following definitions, valid for any convenient vector space $W$ and any $k\geq 0$:
\begin{align}
    W^{-k} &:= \innerhom{W^{\hotimes_\beta k}}{\RR}\\
    \Sym^{-k} W &:=\innerhomsym{W^{\hotimes_\beta k}}{\RR}{k};
\end{align}
here $S_k$ denotes the symmetric group on $k$ elements, and the subscript $(\cdot)_{S_k}$ denotes the space of \emph{coinvariants} for the $S_k$ action.

\begin{corA}
There are canonical isomorphisms
\begin{align}
&\Sym^{-k}(\sV_L) \cong \frac{\Sym^{-k}(\sV)}{\innerhomsym{\sC\hotimes_\beta(\sV)^{\hotimes_\beta (k-1)}}{\RR}{k}}\\
&\Sym^{-k}(\sV_{L,c}) \cong \frac{\Sym^{-k}(\sV_c)}{\innerhomsym{\sC_c\hotimes_\beta(\sV_c)^{\hotimes_\beta (k-1)}}{\RR}{k}}
\end{align}
of convenient vector spaces.
\end{corA}

Note that Corollary A tells us the following about the relationship between $\Sym^{-k}(\sV)$ and $\Sym^{-k}(\sV_L)$.
First, there is a surjective map
\[
\Sym^{-k}(\sV)\to \Sym^{-k}(\sV_L).
\]
Second, the kernel of this map consists precisely of the $S_k$-orbits of maps of the form
\begin{equation}
    \Phi(v_1,\ldots, v_k) = \Phi'(Pv_1,v_2,\ldots, v_k), 
\end{equation}
where 
\[
\Phi' \in \innerhom{\sC\hotimes_\beta \sV^{\hotimes_\beta (k-1)}}{\RR}.
\]
Finally, $\Sym^{-k}(\sV_L)$ is identified with the quotient of the space $\Sym^{-k}(\sV)$ by this kernel, and this identification as a quotient is true both in $\CVS$ and $\DVS$.
This description of $\Sym^{-k}(\sV_L)$ will be useful to us in the body of the dissertation.